\documentclass[11pt]{article}
\usepackage[margin=1in]{geometry} 

\usepackage{amsthm,amsmath,amsfonts}
\usepackage{tikz-cd}
\usepackage{enumerate}	
\usepackage{mathtools}
\usepackage[noadjust]{cite}

\usepackage{commath}
\usepackage{amssymb}
\usepackage{showlabels}
\usepackage{comment}

\usepackage{footnote}
\usepackage{etoolbox}

\newcommand{\defeq}{\vcentcolon=}

\usepackage[utf8]{inputenc}

\usepackage{cite}

\usepackage{nameref,hyperref}

\usepackage{microtype}
\DisableLigatures[f]{encoding = *, family = * }

\usepackage{changepage}

\usepackage[aboveskip=1pt,labelfont=bf,labelsep=period,singlelinecheck=off]{caption}

\makeatletter
\renewcommand{\@biblabel}[1]{\quad#1.}
\makeatother

\usepackage{lastpage,fancyhdr,graphicx}
\usepackage{epstopdf}
\fancyheadoffset[L]{2.25in}
\fancyfootoffset[L]{2.25in}

\usepackage{color}

\definecolor{Gray}{gray}{.25}

\usepackage{graphicx}

\usepackage{sidecap}

\usepackage{wrapfig}
\usepackage[pscoord]{eso-pic}
\usepackage[fulladjust]{marginnote}
\reversemarginpar

\providecommand{\keywords}[1]
{
  \small	
  \textbf{\textit{Keywords---}} #1
}

\theoremstyle{plain}
\newtheorem{theorem}{Theorem}[section] 

\theoremstyle{definition}
\newtheorem{definition}[theorem]{Definition} 
\newtheorem{example}[theorem]{Example} 
\newtheorem{proposition}[theorem]{Proposition} 
\newtheorem{lemma}[theorem]{Lemma}
\newtheorem{corollary}[theorem]{Corollary}

\newtheorem*{claim*}{Claim}
\newtheorem{remark}{Remark}
\newtheorem{remarknot}{Remark and Notation}
\newtheorem*{notation*}{Notation}

\newtheorem{fact}[theorem]{Fact}

\begin{document}


\title{von Neumann regular Hyperrings and applications to Real Reduced Multirings \footnote{The first author was supported by Coordenação de Aperfeiçoamento de Pessoal de Nível Superior (CAPES) and Conselho Nacional de Desenvolvimento Científico e Tecnológico (CNPQ).}}

  \author{ 
    {\large Hugo Rafael de Oliveira Ribeiro}\thanks{hugorafael@usp.br (Corresponding Author)} \\ {\small University of São Paulo, Institute of Mathematics and Statistic (IME), São Paulo, Brasil} \\ 
    {\large  Hugo Luiz Mariano}\thanks{hugomar@ime.usp.br}\\ 
    {\small University of São Paulo, Institute of Mathematics and Statistic (IME), São Paulo, Brasil} 
  }
\date{\today}

  \maketitle

\begin{abstract}
A multiring (\cite{Mar3}) is a kind of ring where is allowed the sum of two elements to be a non-empty subset of the structure instead of just one element -and an hyperring is a multiring with a strong distributive property. Thus a reduced hyperring where the prime spec is a Boolean topological space is called von Neumann regular hyperring (vNH). It is possible to associate to every such object a structural presheaf in the same way it is made with rings but there are some vNH such that this presheaf is not a sheaf. In this sense, we give a first-order characterization of vNH with a structural sheaf (geometric vNH or just GvNH) and how to transform a vNH in a GvNH -in fact, this transformation shows that the category \textbf{GvNH} is a reflexive subcategory of \textbf{vNH}. We also build a von Neumann regular hull for multirings and use this to give applications for algebraic theory of quadratic forms. More precisely, we work with Real Reduced Multiring (RRM, \cite{Mar3}) -also known as Real Semigroup (RS, \cite{DP1})-, a special kind of multirings that is useful to explore the real structure of rings, and show that a von Neumann hull of a RRM is again a RRM. This gives a generalization of sheafs arguments present in \cite{DM4}.
\end{abstract}



Data sharing not applicable to this article as no datasets were generated or analysed during the current study.

\section*{Introduction}


In 2006, Marshall presented in \cite{Mar3} the notion of multiring in the setting of quadratic form theory, a structure like a ring with multivaluated sum. This concept enables a unified way to deal with rings, Special Group, Real Semigroup and several other structures from Algebraic Theory of Quadratic Forms (\cite{RRM}). Besides that, in \cite{Mar3} is given a simple characterization of RS named Real Reduced Multiring (RRM) -more precisely, in \cite{RRM} is described an equivalence between the categories $\textit{RS}$, $\textit{RRM}$ and them are dual to $\textbf{ARS}$ (Abstract Real Spectra).

Also in \cite{Mar3} is proved that for each multiring $A$ can be associated a "canonical" RRM $Q(A)$ -in fact, we proved that this RRM is the best in the sense that the natural projection $\pi_A \colon A \to Q(A)$ is initial (Theorem $\ref{uniProMRR}$). It is an open question to determine which class of RRM we can represent with rings by the functor $Q$.

In \cite{Mar2} (section 8.8), Marshall shows, in the dual language of $\textbf{ARS}$, an example of RRM that is not of the form $Q(A)$ for any ring $A$. An open problem in this sense is given a Real Reduced Hyperfield $M$ (equivalently Reduced Special Group or Abstract Order Space) if exist a ring $A$ such that $M \cong Q(A)$ (or if exist a field $K$ with $M \cong Q(K)$). In \cite{DM3} exist a partial result: every RSG is realized as quotient of a multiplicative subset in the ring of continuous real-valued functions over some Boolean space. The  question remains open also for a Real Reduced Hyperring $M$ but as corollary of Theorem $\ref{QVVQ}$ we show that if exist a semi-real ring $A$ with $Q(A) \cong M$, then exists a semi-real von Neumann regular ring $A'$ such that $M \cong Q(A')$.\\

\textbf{Overview of the paper:}\\

We will start section 1 with some definitions and constructions with concerning multirings, emphasizing the universal property and adaptations concerning the prime spectrum. In section 2 we describe the structural presheaf associated to every multiring. In section 3 we present the notion of von Neumann regular hyperring (vNH) and characterize when it is geometric (that is, the structural pre-sheaf is sheaf) (Theorem $\ref{geoVon}$). A similar work was done in \cite{Jun} where is proved that hyperdomains are geometric and in \cite{DP2} is proved that Real Reduced Multiring (in Real Semigroup language) are also geometric. Furtheremore, we describe a geometric hull for vNH (Corollary $\ref{geoHullvN}$) and give a new representation for the Real Reduced Hull of a GvNH (Theorem $\ref{reprevN}$) using Marshall quotient. We finalize the results with section 4 introducing the von Neumann regular hull of a multiring (Theorem $\ref{uniPGvNH}$). In section 5 we discuss the future works.

\section{Preliminaries}

In this section we give a complete list of construction and related results probably known and very important to what follows. The proofs are given for the reader convenience.

\begin{definition} \label{multiDef}
 A multiring is a tuple $(A, +, -, \cdot, 0, 1)$ where $(A, \cdot, 1)$ is a commutative semigroup, $+ \colon A^2 \to \wp(A)^* \defeq \wp(A)\setminus\{\emptyset\}$ and $- \colon A \to A$
 are functions and $0, 1 \in A$ are constants such that for all $a,b,c \in R$:
 
 \begin{enumerate}[i)]
  \item $a \in b + c \Rightarrow c \in -b + a$ and $b \in a + (-c)$.
  \item $a \in b + 0 \Leftrightarrow a = b$.
  \item (Associativity) If $x \in g + c$ with $g \in a + b$, then exist $h \in b + c$ such that $x \in a + h$.
  \item $a + b = b + a$.
  \item $a \cdot 0 = 0$.
  \item (Half distributivity) If $a \in b + c$, then for all $d \in A$, $ad \in bd + cd$.
 \end{enumerate}
 
\end{definition}

In some multirings, as the usual rings, the property vi) admits a reciprocal

\begin{equation*}
 x \in bd + cd \Rightarrow \mbox{ exist } a \in b + c \mbox{ such that } x = ad.
\end{equation*}

The multirings that satisfies this are called hyperrings.

Some notions for rings have useful generalization to our context. A multiring $A$ is called multidomain if for all $a,b \in A$ with $ab=0$ 
we have $a = 0$ or $b = 0$. 

We define the set of invertible elements $A^{\times} = \{a \in A \colon \mbox{exists } b \in A \mbox{ such that } ab = 1\}$ and the set of weak-invertible elements $A^{\times_w} = \{a \in A \colon \mbox{exists } b_1, \ldots, b_n \in A \mbox{ such that } 1 \in ab_1 + \cdots + ab_n\}$. Note that $A^{\times} \subseteq A^{\times_w}$ and if $A$ is an hyperring, then $A^{\times} = A^{\times_w}$. A multiring $A$ is called multifield if $1 \neq 0$ and all non-zero element is weak-invertible ($A
^{\times_w} = A \setminus \{0\}$) and $A$ is called hyperfield if $1 \neq 0$ and all non-zero element is invertible ($A^{\times} = A \setminus \{0\}$).

\begin{remark}
\begin{itemize}
 \item Note that if $A$ is a hyperfield, then $A$ is an hyperring because given $x \in bd + cd$, if $d = 0$, take $a = 0$; if $d \neq 0$, take $a =x d^{-1}$; in both cases, $x = ad$ and $a \in b + c$. Because of this, we will use the term hyperfield instead of the term multifield used in \cite{Mar3}. 
 
 \item It is possible to define multiring as first-order structure by interpretate the multivalued sum as ternary relation: given a multiring $(A, +, -, \cdot, 0, 1)$, interpret $+$ as a relation $\pi \subseteq A^3$ such that $c \in a + b$ iff $(a,b,c) \in \pi$. Thus every axiom is a Horn-geometric sentence.
 
 \item In the definition of multiring, the multivaluated sum is a priori non-empty but this follows by the axioms. In fact, given $a, b \in A$, since $a \in 0 + a$ and $0 \in b + (-b)$, associativity implies the existence of $h \in a + b$ with $a \in -b + h$.
 \end{itemize}
\end{remark}

\begin{lemma}
 Let $A$ be a multiring. Given $a,b,c \in A$
 
 \begin{enumerate}[i)]
  \item $-0 = 0$ and $-(-a) = a$.
  \item $a \in b + c$ if, and only if, $-a \in -b + -c$.
 \end{enumerate}

\end{lemma}

\begin{proof}
 \begin{enumerate}[i)]
  \item $0 \in 0 + 0 \Rightarrow 0 \in -0 + 0 \Rightarrow -0 = 0$ and $x \in x + 0 \Rightarrow 0 \in -x + x \Rightarrow x \in -(-x) + 0 \Rightarrow -(-x) = x$.
  
  \item $a \in b + c \Leftrightarrow c \in -b + a \Leftrightarrow -b \in c + -a \Leftrightarrow -a \in -b + -c$.
 \end{enumerate}

\end{proof}

\begin{example}
\begin{itemize}
    \item Given a ring $(A, +, -, \cdot, 0, 1)$, define $+' \colon A \times A \to \mathcal{P}(A)^*$ as $a +' b = \{a + b\}$. Then $(A, +', -, \cdot, 0, 1)$ is a multiring.
    
    \item The Krasner hyperfield $\mathbb{K} = \{0,1\}$ is given by usual multiplication and multivalued sum by $0 + 0 = \{0\}, 0 + 1 = \{1\}, 1 + 1 = \{0,1\}$.
    
    \item The signal hyperfield $\textbf{3} = \{1,0,-1\}$ is given by usual multiplication and multivalued sum by $0 + x = \{x\} \ \forall x \in \textbf{3}$, $1 + 1 = \{1\}, -1 + (-1) = \{-1\}, 1 + (-1) = \{1, 0, -1\}$.
\end{itemize}
\end{example}

\begin{definition}
 Let $f \colon A \to B$ be a map between multirings. The function $f$ is a morphism of multirings if for all $a,b,c \in A$:
 
 \begin{enumerate}[i)]
  \item If $a \in b +_A c$, then $f(a) \in f(b) +_B f(c)$.
  \item $f(a \cdot_A b) = f(a) \cdot_B f(b)$.
  \item $f(0_A) = 0_B, f(1_A) = 1_B$.
  \item $f (-_A a) = -_B f(a)$.
 \end{enumerate}

\end{definition}

Given $A$ a multiring and $S, T \subseteq A$, it is defined $S + T = \bigcup_{s \in S, t \in T} s + t$ and $S \cdot T = \{st \colon s \in S, t \in T\}$. Given $x_1, \ldots, x_n \in A$, we define by induction $x_1 + \cdots + x_n = \{x_1\} + (x_2 + \cdots + x_n)$.
A set $\alpha \subseteq A$ is an ideal if $\alpha + \alpha \subseteq \alpha$ and $A \alpha \subseteq \alpha$ and $\alpha$ is a proper ideal
if $1 \notin \alpha$. An ideal $\alpha$ is prime if it is a proper ideal and given $a,b \in A$ such that $ab \in \alpha$ then $a \in \alpha$ or $b \in \alpha$ and $\alpha$
is maximal if it is maximal between all proper ideals. An equivalent way to define prime ideals is through morphism to the Krasner hyperfield: if $f \colon A \to \mathbb{K}$ is a morphism, then $f^{-1} (0)$ is a prime ideal and given $p \subseteq A$ prime ideal, the characteristic map $\chi_p \colon A \to \mathbb{K}$ given by $\chi_p(x) = 0$ iff $x \in p$ is a multiring morphism. The set of all prime ideals are denoted by $\mbox{spec}(A)$. 
A set $S \subseteq A$ is multiplicative if $1 \in S$ and $S \cdot S \subseteq S$.
Given a subset $X \subseteq A$, the ideal generated by $X$ is $\bigcup \{t_1x_1 + \cdots + t_n x_n \colon  n \geq 1, t_i \in A, x_i \in X \mbox{ for all } i=1, \ldots, n\}$. Note
that if $X = \{x_1, \ldots, x_n\}$ is finite, in general $I = \bigcup \{t_1x_1 + \cdots + t_nx_n \colon t_i \in A \mbox{ for all } i=1, \ldots, n\}$ is not the ideal generated
by $X$ (it is possible not closed by addition) but if $A$ is an hyperring this is true.

\begin{theorem} \label{PIT}
 Let $A$ be a multiring.
 
 \begin{enumerate}[i)]
  \item Every maximal ideal is a prime ideal.
  \item (Prime Ideal Theorem) Let $I \subseteq A$ an ideal and $S \subseteq A$ a multiplicative set such that $I \cap S = \emptyset$. Then exists a prime ideal $p$ such that $I \subseteq p$ and $p \cap S = \emptyset$. In particular, $A \neq 0$ if, and only if, $\mbox{spec}(A) \neq \emptyset$.
  \item Given $a \in A$, define $D(a) = \{P \in \mbox{spec}(A) \colon a \notin P\}$. Then $D(a) \cap D(b) = D(ab)$ and the family $\{D(a) \colon a \in A\}$
  is a basis for a spectral topology in $\mbox{spec}(A)$.
 \end{enumerate}

\end{theorem}

\begin{proof}
 
 \begin{enumerate}[i)]
  \item Let $\mathfrak{m} \subseteq A$ be a maximal ideal and $a,b \in A$ such that $ab \in \mathfrak{m}$. If $a \notin \mathfrak{m}$, consider the ideal $I$ generated by $\mathfrak{m} \cup \{a\}$
  
  \begin{align*}
   I = \bigcup \{m + t_1a + \cdots + t_na \colon n \geq 1, m \in \mathfrak{m}, t_i \in A \mbox{ for all } i=1, \ldots, n\}.
  \end{align*}

  Since $\mathfrak{m}$ is maximal ideal and $a \in I \setminus m$, $1 \in I$. So exist $m \in \mathfrak{m}$ and $t_1, \ldots, t_n \in A$ such that $1 \in m + t_1a + \cdots + t_na$.
  Then $b \in bm + t_1ab + \cdots t_nab \subseteq \mathfrak{m}$.

 \item Let $X = \{J \subseteq A \colon J \mbox{ is an ideal with } I \subseteq J, J \cap S = \emptyset \}$. With the order given by inclusion, $(X, \subseteq)$ is a non-empty partial order ($I \in X$). It is straightforward to conclude by Zorn's lemma that exists $p \in X$ maximal. If $p$ is not prime, exists $a,b \in A$ such that $ab \in p$ and $a,b \notin p$. Consider the ideal $J = p + (a) = \bigcup \{x + at_1 + \cdots + at_n \colon x \in p, t_1, \ldots, t_n \in A\}$. Since $a \in J \setminus p$ and $p \subseteq J$, by $p$ maximality we have $J \cap S \neq \emptyset$. Then exists $s_1 \in S, x \in p$ and $t_1, \ldots, t_n \in A$ such that $s_1 \in x + at_1 + \cdots +at_n$. Using the ideal $p + (b)$, also exists $s_2 \in S, y \in p$ and $l_1, \ldots, l_k \in A$ such that $s_2 \in y + bl_1 + \cdots + bl_n$. Then $s_1 s_2 \in xy + xbl_1 + \cdots + xbl_n + yat_1 + \cdots yat_n + \sum_{i,j} abt_il_j \subseteq S \cap p$, an absurd. Thus $p$ is prime.

 \item The proof is analogous to the ring case and can be found in Proposition 2.3 of \cite{Mar3}.
 \end{enumerate}

\end{proof}

\begin{remark} \label{remarkOpen}

\begin{itemize}
    \item 
If $f \colon A \to B$ is a morphism between multirings, for each $p \in \mbox{spec}(B)$, $f^{-1}(p) \in \mbox{spec}(A)$ is a prime ideal. This induces a map $f^* = f^{-1} \colon \mbox{spec}(B) \to \mbox{spec}(A)$ such that for each $a \in A$, $(f^*)^{-1} (D_A(a)) = D_B(f(a))$ and $f^* (D_B(f(a))) = D_A(a) \cap \mbox{Im}(f^{*})$. In particular, $f^{*}$ is a spectral\footnote{Given $X,Y$ spectral topological spaces, a function $f \colon X \to Y$ is spectral if pre-image of compact open is compact open. In particular, spectral maps are continuous and a continous map between Boolean topological spaces is spectral.} map and if $f$ is surjective, $f^* \colon \mbox{spec}(B) \to \mbox{Im}(f^*) \subseteq \mbox{spec}(A)$ is open. 

\item Let $A$ be a multiring and $p \in \mbox{spec}(A)$. Then $\overline{\{p\}} = \{q \in \mbox{spec}(A) \colon p \subseteq q\}$. In fact, given $q \in \mbox{spec}(A)$ with $p \subseteq q$ and $x \in A$ with $q \in D(x)$, then $p \in D(x)$ and thus $q \in \overline{\{p\}}$. Reciprocally, if $q \in \overline{\{p\}}$, given $x \notin q$, that is, $q \in D(x)$, we have $p \in D(x)$. Thus $p \subseteq q$. In particular, $p \in \mbox{spec}(A)$ is maximal if, and only if, $\{p\} \subseteq \mbox{spec}(A)$ is closed.
\end{itemize}

\end{remark}

 In what follows, $A$ is a multiring, $S \subseteq A$ is a multiplicative set and $I \subseteq A$
 is an ideal.  We sumarize some basic constructions with multirings.\\
 

 \textbf{Quocient by ideal.} Given $a, b \in A$, $a$ and $b$ are equivalent module $I$ (notation $a \sim_I b$) if $a - b \cap I \neq \emptyset$.
 Let $A / I$ the set of all equivalence classes of $\sim_I$. The multivaluated sum is given by $\overline{a} \in \overline{b} + \overline{c}$ if,
 and only if, exists $a', b', c' \in A$ with $a' \sim_I a$, $b' \sim_I b$ and $c' \sim_I c$ such that $a' \in b' + c'$. This relation can be described
 by $\overline{a} \in \overline{b} + \overline{c}$ if, and only if, exist $i \in I$ such that $a \in b + c + i$. The product is $\overline{a} \cdot \overline{b} = \overline{ab}$ and $- \overline{a} = \overline{-a}$. 
 The constants are induced by those of $A$. Then $A / I$ is a multiring and exist a canonical projection $\pi_I: A \to A / I$ which is a morphism of multirings with $\pi_I(a) = 0$ if, and only if, $a \in I$.
 
 \begin{proposition} \label{idealq}
  Let $A$ be a multiring and $I$ an ideal.
  
  \begin{enumerate}[i)]
   \item Let $a,b,c \in A$. Then $\pi(a) \in \pi(b) + \pi(c)$ if, and only if, exist $a' \in A$ such that $a' \sim_I a$ and $a' \in b + c$. \label{78}
   \item If $A$ is hyperring, then $A / I$ is hyperring.
   \item $I$ is prime if, and only if, $A / I$ is a multidomain. \label{idealqPrime}
   \item The ideal $I$ is maximal if, and only if, $A / I$ is a multifield. In particular, if $A$ is an hyperring, $I$ is maximal if, and only if, $A/I$ is a hyperfield. \label{idealqMax}

   \item Given a multiring morphism $f \colon A \to B$ with $f(I) = \{0\}$, then exists an unique morphism $\overline{f} \colon A / I \to B$ such that $f = \overline{f} \circ \pi$.
   
   \item The induced spectral map $\pi^{-1} \colon \mbox{spec}(A / I) \to \mbox{spec}(A)$ determines an homeomorphism between $\mbox{spec}(A/I)$ and $\{P \in \mbox{spec}(A) \colon I \subseteq P\}$.
   
   \item Let $f \colon A \to B$ a multiring morphism and $J \subseteq B$ an ideal. Given $I \subseteq f^{-1}(J)$ ideal, then exists a unique morphism $f_{I, J} \colon A / I \to B / J$
   such that

     \[
  \begin{tikzpicture}
  \matrix (m) [matrix of math nodes,row sep=3em,column sep=4em,minimum width=2em]
  {
     A & B \\
     A/I & B/J \\};
  \path[-stealth]
    (m-1-1) edge node [above]{$f$} (m-1-2)
	    edge node {} (m-2-1)
    (m-1-2)   edge node {} (m-2-2)
    (m-2-1) edge node [above]{$f_{I,J}$} (m-2-2);
  \end{tikzpicture}
  \]

   is a commutative diagram.
  \end{enumerate}

 \end{proposition}

\begin{proof}
 
 \begin{enumerate}[i)]

  \item If $\pi(a) \in \pi(b) + \pi (c)$, we know that exist $i \in I$ such that $a \in b + c + i$. Then $-b \in c + (-a + i)$. So exist $-a' \in -a + i$ such that $-b \in c - a'$ and
  then $a' \in b + c$ and $a' \sim_I a$.
  
  \item Assume that $A$ is an hyperring and let $a,b,c \in A$. Take $\pi(x) \in \pi(ca) + \pi(cb)$. Then by $\ref{78})$ exists $x' \in A$ with $x' \sim_I x$ and $x' \in ca + cb$. Since $A$
  is an hyperring, exist $y \in a+b$ such that $x' = cy$. Then $\pi(y) \in \pi(a) + \pi(b)$ and $\pi(x) = \pi(x') = \pi(cy)$.
  
  \item Analogous to the ring case.
  
  \item Assume that $I$ is maximal and let $a \notin I$ ($a \neq 0$ in $A/I$). Noting that the ideal generated by $I \cup \{a\}$ is $\bigcup \{i + x_1a + \cdots + x_na \colon n \geq 1, i \in I, x_1, \ldots, x_n \in A\}$, by $I$ maximality exists $i \in I$ and $x_1, \ldots, x_n \in A$ such that $1 \in i + x_1 a + \cdots + x_na$. Thus $1 \in x_1a + \cdots + x_na$ in $A/I$. The reciprocal follows by observing that if the ideal $I$ has the desired property, then for every $a \notin I$, the ideal generate by $I \cup \{a\}$ is improper.

 \item Let $f \colon A \to B$ such that $f(I) = \{0\}$. Note that
 
 \begin{itemize}
  \item If $\pi(a) = \pi(b)$, exist $i \in I$ such that $a \in b+ i$. Then $f(a) \in f(b) + 0 = \{f(b)\}$.
  
  \item If $\pi(a) \in \pi(b) + \pi(c)$, then exist $a' \in A$ such that $\pi(a) = \pi(a')$ and $a' \in b+c$. Then $f(a) = f(a') \in f(b) + f(c)$.
 \end{itemize}

 Thus, we can define $\overline{f}(\pi(a)) = f(a)$, which by previous observations is a multiring morphism. The uniqueness of $\overline{f}$ is trivial.

\item The bijection between $\mbox{spec}(A / I)$ and $\{p \in \mbox{spec}(A) \colon I \subseteq A\}$ follows by the universal property of $\pi$ using the equivalent characterization of prime ideals as morphisms to Krasner hyperfield. By Remark $\ref{remarkOpen}$, $\pi^{-1}$ is an homeomorphism with the image.

 \item Let $f \colon A \to B$ a multiring morphism, $J \subseteq B$ ideal and $I \subseteq f^{-1}(J)$. Consider the map $\pi_J \colon B \to B/J$. Since $\pi_J \circ f (I) = 0$, by the universal
 property for $\pi_I \colon A \to A/I$, exist unique multiring morphism $f_{I,J} \colon A / I \to B/J$ such that $f_{I,J} \circ \pi_I = \pi_J \circ f$.
 \end{enumerate}

\end{proof} 

\begin{example}
\begin{itemize}
 \item The signal function $\mbox{sgn} \colon \mathbb{R} \to \textbf{3}$ is an example of non-injective morphism that satisfies $\mbox{sgn}(x) = 0 \Leftrightarrow x = 0$. On the other hand, if $f \colon A \to B$ is surjective morphism satisfying the property in $\ref{idealq}, \ref{78})$, that is, for $x,y, z \in A$,
 
 \begin{equation*}
     \mbox{if }f(x) \in f(y) + f(z), \mbox{ then exists }x' \in A \mbox{ such that } f(x') = f(x) \mbox{ and } x' \in y + z,
 \end{equation*}
 
 then the induced morphism $\overline{f} \colon A / I \to B$ is an isomorphism.
 
 \item A multiring $D$ satisfies a strong cancellative property if for all $a,b,c \in D$ with $ab = ac$ and $a \neq 0$, then $a = b$. There are multidomains $D$ that not satisfies strong cancellative property. For example, let $R$ be a RRM (see, for instance, Definition $\ref{rrmdefinition}$) that has a prime ideal $p$ that is not maximal. Then $R / p$ is a multidomain ($\ref{idealq}, \ref{idealqPrime})$) and satisfies $x^3 = x$ for all $x \in R/p$. Thus if $R/p$ would satisfy the strong cancellative property, then $R/P$ would be a hyperfield and by $\ref{idealq}, \ref{idealqMax})$ $p \in \mbox{spec}(R)$ should be a maximal ideal, a contradiction.
 \end{itemize}
 
\end{example}

\begin{remark}
\begin{itemize}
    \item The map $f_{I,J}$ is usually denoted only by $f_J$ when $I = f^{-1}(J)$.
\end{itemize}

\end{remark}

\textbf{Localization.} The elements of $S^{-1} A$ are of the form $a/s$ with $a \in A$ and $s \in S$ and $a/s = b/t$ if, and only if, exist $u \in S$ such that
$atu = bsu$. The sum is defined by $a/s \in b/t + c/u$ if, and only if, exist $v \in S$ such that $atuv \in bsuv + cstv$. The product is $a/s \cdot b/t \coloneqq ab/st$.
The unit element is $1 / 1$ and the zero element is $0/1$. All those are well-defined and make $S^{-1}A$ into a multiring. The canonical map $\rho_S \colon A \to S^{-1}A$
given by $\rho_S (a) = a / 1$ is a morphism and $\rho_S (S) \subseteq (S^{-1}A)^{\times}$.

 \begin{proposition} \label{localization}
  Let $A$ be a multiring and $S$ a multiplicative set.
  
  \begin{enumerate}[i)]
   \item If $A$ is hyperring, then $S^{-1}A$ is hyperring.
   \item $S^{-1}A = 0$ if, and only if, $0 \in S$.
   
   \item Given a multiring morphism $f \colon A \to B$ with $f(S) \subseteq B^{\times}$, exist unique morphism $\overline{f} \colon S^{-1}A \to B$ such that $f = \overline{f} \circ \rho_S$.
   In particular, $\rho_S$ is an epimorphism.
   
   \item The induced spectral map $\rho^{-1} \colon \mbox{spec}(S^{-1}A) \to \mbox{spec}(A)$ determines an homeomorphism between $\mbox{spec}(S^{-1}A)$ and $\{P \in \mbox{spec}(A) \colon P \cap S = \emptyset\}$.
   
   \item Let $f \colon A \to B$ a multiring morphism, $T \subseteq B$ a multiplicative set and $S \subseteq f^{-1}(T)$ another multiplicative set. Then exists an unique morphism $f_{S,T} \colon S^{-1}A \to T^{-1}B$
   such that
      \[
  \begin{tikzpicture}
  \matrix (m) [matrix of math nodes,row sep=3em,column sep=4em,minimum width=2em]
  {
     A & B \\
     S^{-1}A & T^{-1}B \\};
  \path[-stealth]
    (m-1-1) edge node [above]{$f$} (m-1-2)
	    edge node {} (m-2-1)
    (m-1-2)   edge node {} (m-2-2)
    (m-2-1) edge node [above]{$f_{S,T}$} (m-2-2);
  \end{tikzpicture}
  \]
  is a commutative diagram.
  \end{enumerate}

 \end{proposition}

\begin{proof}
 
 \begin{enumerate}[i)]
  \item Let $x,a,b,c \in A$ and $s,u,t,w \in S$ such that $x/s \in c/w (a/u) + c/w (b/t) = ca/wu + cb/wt$. Then exist $p \in S$ such that $xputw^2 \in ca p stw + cb p suw$.
  Since $A$ is hyperring, exist $d \in a t + b u$ such that $xputw^2 = d (scwp)$. Then $d/tu \in a/u + b/t$ and $x/s = d/tu \cdot c/w$. So $S^{-1}A$ is hyperring.
  
  \item If $0 \in S$, it is obvious that $S^{-1}A = 0$. Reciprocally, if $S^{-1}A = 0$, then $1 = 0$ in $S^{-1}A$, that is, exists $s \in S$ such that $s = s \cdot 1 = s \cdot 0 = 0$.

  \item Let $f \colon A \to B$ with $f(S) \subseteq B^{\times}$. Note that
  
  \begin{claim*}
   If $a/s \in b/t + c/u$, then $f(a)f(s)^{-1} \in f(b)f(t)^{-1} + f(c)f(u)^{-1}$. In particular, if $a/s = b/t$, then $f(a)f(s)^{-1} = f(b) f(t)^{-1}$.
   \end{claim*}
   
   \begin{proof}
   Take $p \in S$ such that $aptu \in bpsu + cpst$. Since $f$ is a morphism, $f(a)f(p)f(t)f(u) \in f(b)f(p)f(s)f(u) + f(c)f(p)f(s)f(t)$. But $f(p),f(s),f(t),f(u) \in B^{\times}$ and so $f(a)f(s)^{-1} \in f(b)f(t)^{-1} + f(c)f(u)^{-1}$.
  \end{proof}
  
  Thus we can define $\overline{f} \colon S^{-1}A \to B$ by $\overline{f}(a/s) = f(a)f(s)^{-1}$  and then by above claim $\overline{f}$ is a multiring morphism. The uniqueness of $\overline{f}$ is trivial.
  
  \item The bijection between $\mbox{spec}(S^{-1}A)$ and $\{p \in \mbox{spec}(A) \colon p \cap S = \emptyset\}$ follows by the universal property of $\rho$ using the equivalent characterization of prime ideals as morphisms to Krasner hyperfield. 
  By Remark $\ref{remarkOpen}$, $\rho^{-1}$ is an homeomorphism with the image.
  
  \item Follows directly by the universal property of $A \to S^{-1}A$.

 \end{enumerate}

\end{proof}

\begin{remark}
 \begin{itemize}
  \item If $A$ is a multidomain, we can define the fraction field of $A$ by $ff(A) \coloneqq (A \setminus 0)^{-1} A$. Note that $ff(A)$ is a hyperfield and the canonical map
$\rho \colon A \to ff(A)$ is not necessarily injective (Example 2.5 in \cite{Mar3}). Given a prime ideal $p \in \mbox{spec}(A)$, we define $K_A(p) \defeq ff (A / p)$ and $A_p = (A \setminus p)^{-1} A$. Let $\rho \colon A \to A_p$ the canonical morphism. Note that, by Proposition $\ref{localization}$, iv), the multiring $A_p$ has an unique maximal ideal given by $A_p \rho(p) = \{\frac{x}{s} \colon x \in p, s \in A \setminus p\}$ and it is denoted by $pA_p$.

\item The morphism $f_{S,T}$ is usually denoted only by $f_T$ when $S = f^{-1}(T)$.
\end{itemize}

\end{remark}

\begin{proposition} \label{kAp}
  Let $A$ a multiring and $p \in \mbox{spec}(A)$ a prime ideal. Then exists an unique map $A_p / pA_p \to K_A(p)$ such that 
 
    \[
  \begin{tikzcd}
    & A \arrow{dr} \arrow{dl} \\ 
  A_p /pA_p \arrow{rr} && K_A(p)
  \end{tikzcd}
  \]
  is a commutative diagram. Furtheremore, it is an isomorphism.
\end{proposition}

\begin{proof}
 Consider the compositions $i_1 \colon A \to A/p \to K_A(p), i_2 \colon A \to A_p \to A_p/pA_p$. Note that, by propositions $\ref{idealq}$ and $\ref{localization}$, they satisfies the following universal property: given a map $f \colon A \to B$ such that
  $f(p)=0$ and $f(A \setminus p) \subseteq B^{\times}$, then exists uniques $f_1 \colon K_A(p) \to B, f_2 \colon A_p/pA_p \to B$ such that $f = f_1 \circ i_1 = f_2 \circ i_2$.
  These universal properties assure that exists unique $f$ with $i_2 = i_1 \circ f$ and $f$ should be an isomorphism.
\end{proof}

\textbf{Marshall Quotient.} Given $a,b \in A$, $a,b$ are said Marshall equivalent (denoted by $a \sim_S b$) if exists $s, t \in S$ such that $as = bt$. The
set of all induced equivalence classes is denoted by $A /_m S$. The multivaluated sum is defined by $\overline{a} \in \overline{b} + \overline{c}$ if, and only if,
exists $a', b', c' \in A$ such that $a' \sim_S a$, $b' \sim_S b$, $c' \sim_S c$ and $a' \in b' + c'$ (note that $\overline{a} \in \overline{b} + \overline{c}$ 
if, and only if, exists $s,t, u \in S$ such that $as \in bt + cu$). The product is defined by $\overline{a} \cdot \overline{b} = \overline{ab}$ and the constants are those induced by $A$.
Then $A /_m S$ is a multiring and exist a canonical map $\pi_S \colon A \to A /_m S$ is a surjective morphism and $\pi_S(S) = \{1\}$.

\begin{definition}
 Let $A$ a multiring and $S \subseteq A$ multiplicative subset. If for all $xs \in S$ with $s \in S$ we have $x \in S$, then $S$ is called cancellative. Define also $\overline{S} = \{x \in A \colon xs \in S \mbox{ for some } s \in S\}$. Note that $S$ is cancellative if, and only if, $S = \overline{S}$.
\end{definition}

\begin{proposition} \label{marshallQ}
  Let $A$ be a multiring and $S$ a multiplicative set.
  
  \begin{enumerate}[i)]
    \item $\overline{S}$ is cancellative multiplicative set and given $T \subseteq A$ multiplicative, $S,T$ induces the same equivalence relation in $A$ if, and only if, $\overline{S} = \overline{T}$. In particular, $A /_m \overline{S} = A /_m S$.
  
   \item If $A$ is hyperring, $A /_m S$ is hyperring.
   
   \item Given a multiring morphism $f \colon A \to B$ with $f(S) = \{1\}$, exists an unique morphism $\overline{f} \colon A /_m S \to B$ such that $f = \overline{f} \circ \rho$.
   
   \item The induced spectral map $\pi^{-1} \colon \mbox{spec}(A /_m S) \to \mbox{spec}(A)$ determines an homeomorphism between $\mbox{spec}(A /_m S)$ and $\{P \in \mbox{spec}(A) \colon P \cap S = \emptyset\}$.
   \item Let $f \colon A \to B$, $T \subseteq B$ a multiplicative set and $S \subseteq f^{-1}(T)$ another multiplicative set. Then exist unique morphism $f_{S,T} \colon A /_m S \to B /_m T$
   such that
      \[
  \begin{tikzpicture}
  \matrix (m) [matrix of math nodes,row sep=3em,column sep=4em,minimum width=2em]
  {
     A & B \\
     A /_m S & B /_m T \\};
  \path[-stealth]
    (m-1-1) edge node [above]{f} (m-1-2)
	    edge node {} (m-2-1)
    (m-1-2)   edge node {} (m-2-2)
    (m-2-1) edge node [above]{$f_{S,T}$} (m-2-2);
  \end{tikzpicture}
  \]
  is a commutative diagram.
  \end{enumerate}

 \end{proposition}

\begin{proof}
 \begin{enumerate}[i)]
  \item First note that
  $\overline{S}$ is multiplicative and cancellative.\\
    Since $S \subseteq \overline{S}$, $1 \in \overline{S}$. Given $x, y \in \overline{S}$, exists $s,t \in S$ such that $xs,yt \in S$. Then $xyst \in S$ and so
    $xy \in \overline{S}$. Also if $xs \in \overline{S}$ with $s \in \overline{S}$, then exists $s_1,s_2 \in S$ such that $xss_1, ss_2 \in S$. Therefore $x(ss_2)s_1 \in S$ and so $x \in \overline{S}$.\\
    
    Now assume that $S,T$ satisfies $\sim_S = \sim_T$. Since $\overline{S} = \{x \in A \colon x \sim_S 1\}$, we have $\overline{S} = \overline{T}$. Reciprocally, if $\overline{S} = \overline{T}$, given $x,y \in A$ with $x \sim_S y$, exists $s,s' \in S$ such that $xs = ys'$. By hypothesis exists $t,t' \in T$ with $st,st' \in T$. Then $xs(tt') = xs'(tt')$ and $x \sim_T y$. Change the roles of $S$ and $T$, we see that $\sim_S = \sim_T$.

  \item Let $x,a,b,c \in A$ such that $\pi(x) \in \pi(ca) + \pi(cb)$. Then exists $s,t,u \in S$ such that $xs \in cat + cbu$. Since $A$ is an hyperring, exist $d \in at + bu$ such that
  $xs = cd$. Then $\pi(x) = \pi(cd)$ and $\pi(d) \in \pi(a) + \pi(b)$. Thus $A/_mS$ is hyperring.
  
  \item Let $f \colon A \to B$ a multiring morphism such that $f(S) \subseteq \{1\}$. The note that
  
  \begin{itemize}
   \item If $\pi(a) \in \pi(b) + \pi(c)$, then $f(a) \in f(b) + f(c)$. In particular, if $\pi(a) = \pi(b)$, then $f(a) = f(b)$.
  \end{itemize}

  Thus, we just need to define $\overline{f} (\pi(a)) = f(a)$: $\overline{f} \colon A /_m S \to B$ is a well-defined multiring morphism. The uniqueness is trivial.
  
    \item The bijection $\pi^{-1} \colon \mbox{spec}(A /_m S) \to \{P \in \mbox{spec}(A) \colon P \cap S = \emptyset\}$ is a direct consequence of the universal property of $\pi \colon A \to A /_m S$ and the characterization of prime ideals as morphisms to Krasner hyperfield. By Remark $\ref{remarkOpen}$, $\pi^{-1}$ is an homeomorphism with the image.

    \item Follows directly by the universal property of $A \to A /_m S$.
 \end{enumerate}

\end{proof}

\begin{remark}
The morphism $f_{S,T}$ is usually denoted only by $f_T$ when $S = f^{-1}(T)$.
\end{remark}

\begin{proposition} 
 Let $A$ a multiring. Let $S$ a multiplicative set and $\pi \colon A \to A /_m S$. Given $q \in \mbox{spec}(A /_m S)$, let $p = \pi^{-1}(q) \in \mbox{spec}(A)$. Consider the canonical map $g \colon A \to K_A(p)$ and let
  $S_p = g(S)$. Then exist unique morphism $K_{A /_m S} (q) \to K_A(p) /_m S_p$ such that
  
   \[
  \begin{tikzcd}
    & A \arrow{dr} \arrow{dl} \\ 
  K_{A /_m S} (q) \arrow{rr}{} && K_A(p)/_m S_p
  \end{tikzcd}
  \]

  is a commutative diagram. Furtheremore, it is an isomorphism.
\end{proposition}

\begin{proof}
The strategy is the same as in Proposition $\ref{kAp}$. Consider the compositions $j_1 \colon A \to A/_m S \to K_{A /_m S}(q), j_2 \colon A \to K_A(p) \to K_A (p) /_m S_p$. They satisfies the following universal property: given a map $f \colon A \to B$ such that
  $f(S) = 1, f(p)=0$ and $f(A \setminus p) \subseteq B^{\times}$, then exists uniques $f_1 \colon K_{A/_mS}(q) \to B, f_2 \colon K_A (p) /_m S_p \to B$ such that $f = f_1 \circ j_1 = f_2 \circ j_2$.
\end{proof}

 \textbf{Inductive limits.} Let $L$ be a language and consider the category $\textbf{L-mod}$ of all $L$-models. Given a right-directed set (or simply directed) $\langle I, \leq \rangle$, an I-inductive system of $L$-models is a functor $\mathcal{M} \colon \langle I, \leq \rangle \to \textbf{L-mod}$ (if $i \leq j$, we denote the morphism $\mathcal{M}(i \leq j)$ by $\mu_{i,j}$). A colimit of such system is denoted by $\varinjlim \mathcal{M}$ or $\varinjlim \mathcal{M}_i$.
 
 A dual cone over the inductive system $\mathcal{M}$ is a tuple $\langle A, \{\mu_i \colon i \in I\}\rangle$ such that $A$ is a $L$-strucutre and $\mu_i \colon \mathcal{M}_i \to A$ is a $L$-morphism with $\mu_i \circ \mu_{j,i} = \mu_j$ for every $j \leq i$. For the convenience of the reader, we summarize the relations between this notions:

\begin{proposition} \label{filtColim}
 Let $\mathcal{M} = \langle I, \leq \rangle \to L-\mbox{mod}$ be a inductive system of $L$-structures. Then the colimit exists in $L$-mod and a dual cone over $\mathcal{M}$, $\langle M, \{\mu_i \colon i \in I\} \rangle$, is isomorphic to $\varinjlim \mathcal{M}$ if, and only if,
    
    \begin{enumerate}[a)]
        \item $M = \bigcup_{i \in I} \mu_i (\mathcal{M}_i)$.
        
        \item If $\varphi (v_1, \ldots, v_n)$ is an atomic $L$-formula and $s_1, \ldots, s_n \in \mathcal{M}_i$ satisfies $M \vDash \varphi(\mu_i(s_1), \ldots, \mu_i(s_n))$, then exists $k \in I$ with $k \geq i$ such that
        
        \begin{align*}
        \mathcal{M}_k \vDash \varphi(\mu_{i,k} (s_1), \ldots, \mu_{i,k}(s_n)).    
        \end{align*}
        
        In particular, given $i,j \in I$ and $a \in \mathcal{M}_i, b \in \mathcal{M}_j$, we have $\mu_i(a) = \mu_j(b)$ if, and only if, exists $k \geq i,j$ such that
        $\mu_{i,k}(a) = \mu_{j,k}(b)$.
    \end{enumerate}
    
    Furthermore, let $\varphi(\overline{v})$ be a disjunction of geometric formulas \footnote{a formula $\varphi (\overline{t})$ is geometric if it is a negation of an atomic formula or of the form $\forall \overline{v} (\psi_1(\overline{v}, \overline{t}) \rightarrow \exists \overline{w} \psi_2(\overline{w}, \overline{v}, \overline{t}))$, where $\psi_1, \psi_2$ are positive and quantifier-free.} and $\overline{s} \in \mathcal{M}_i$. Let $S_{\varphi} = \{k \in I \colon k \geq i \mbox{ and } \mathcal{M}_k \models \varphi(\mu_{i,k}(\overline{s})) \}$. If $S_{\varphi}$ is cofinal in $I$, then $M \models \varphi(\mu_i(\overline{s}))$.
\end{proposition}

\begin{proof}
The result can be found in \cite{Mir1} or \cite{DM4}.
\end{proof}

\begin{remark}
 Let $\mathcal{L}_{multi} = \{\pi, \cdot, -, 0, 1\}$ be the language of multirings where $\pi$ is ternary relation, $\cdot$ binary function, $-$ unary function and $0,1$ constants (here the multivaluated sum is a ternary relation $\pi$: $a \in b + c$ if, and only if, $\pi(a,b,c)$). Let $\mathcal{M} \colon \langle I, \leq \rangle \to \textbf{Multi}$ be a inductive system of multirings (for each $i \in I$, denote $\mathcal{M}(i) = M_i$). 
 
 Then $\varinjlim \mathcal{M}$ is a multiring because all axioms of $\ref{multiDef}$ are geometrics sentences.
\end{remark}

Let $A$ be a multiring. A prime cone of $A$ is just a subset $P \subseteq A$ such that $A^2 \subseteq P$, $P+P \subseteq P, P \cdot P \subset P, P \cup -P = A$ and $\mbox{supp}(P) \defeq P \cap -P$ (support of $P$) is a prime ideal (when $A$ is an hyperfield, $\mbox{P} = \{0\}$, the unique prime ideal).
An order of $A$ is just a morphism $\sigma \colon A \to 3$ and the set of all orders is denoted by $\mbox{sper}(A)$ and is called real spectra. These data are equivalent in the following sense:

\begin{proposition} 
 Let $A$ be a multiring. Then the following data are in bijective correspondence:
 
 \begin{enumerate}[i)]
  \item $PC = \{P \subseteq A \colon P\mbox{ is a prime cone}\}$.
  \item $PQ = \{(p,P) \colon \mbox{ where } p \in \mbox{spec}(A) \mbox{ and } P \mbox{ is a prime cone of } K_A(p) \}$.
  \item $\mbox{sper}(A) = \{ \sigma \colon A \to 3 \colon \sigma \mbox{ is a morphism} \}$.
 \end{enumerate}

\end{proposition}

\begin{proof}
The bijections are described in terms of $\mbox{sper}(A)$.
 
 $i) \Leftrightarrow iii)$: Given a prime cone $P \subseteq A$, consider the morphism $\sigma \colon A \to 3$ given by
 
 \begin{equation}
    \sigma(a) =
    \begin{cases*}
      1 & if $a \in P \setminus -P$ \\
      0 & if $a \in P \cap - P$ \\
      -1 & if $a \in -P \setminus P$.
    \end{cases*}
  \end{equation}\\
 
 Reciprocally, if $\sigma \in \mbox{sper}(A)$, then $\sigma^{-1} (\{0,1\}) \subseteq A$ is a prime cone.\\
 
 $ii) \Leftrightarrow iii)$: Given a prime ideal $p \in \mbox{spec}(A)$ and a prime cone $P \subseteq K_A(p)$, consider the morphism $\sigma \colon A \to 3$

 \begin{equation}
    \sigma(a) =
    \begin{cases*}
      1 & if $a \notin p$ and $\overline{a} \in P$ \\
      0 & if $a \in p$ \\
      -1 & if $a \notin p$ and $-\overline{a} \in P$.
    \end{cases*}
  \end{equation}\\
 
 Reciprocally, if $\sigma \colon A \to 3$ is a morphism, then $\sigma^{-1}(0)$ and the image of $\sigma^{-1}(\{0,1\})$ by the canonical composition $A \to A/p \to K_A(p)$ gives a prime cone in $K_A(p)$.
\end{proof}

Let $A$ be a multiring. Given $a_1, \ldots, a_n \in A$, define $U(a_1, \ldots, a_n) = \{\sigma \in \mbox{sper}(R) \colon \sigma(a_i) = 1 \mbox{ for all } i=1, \ldots, n\}$. The multiring $A$ is called semi-real if $-1 \notin \sum A^2$ and
an ideal $\alpha \subseteq A$ is real if given $a_1, \ldots, a_n \in A$ such that $a_1^2 + \cdots + a_n^2 \cap A \neq \emptyset$, then $a_i \in A$ for all $i=1, \ldots, n$.

\begin{proposition} \label{basicReal}
 Let $A$ be a multiring.
 
 \begin{enumerate}[i)]
  \item $\mbox{sper}(A) \neq \emptyset$ if, and only if, $A$ is semi-real.
  \item The set $\mbox{sper}(A)$ endowed with the topology generated by the sets $U(a_1, \ldots, a_n)$, $a_i \in A$, is a spectral space. 
  \item Let $p$ a prime ideal of $A$. Then the following are equivalent:
  \begin{enumerate}[a)]
  \item $p$ is real.
  \item The multfield $K_A(p)$ is semi-real.
  \item $p$ is the support of some order.
  \end{enumerate}
 \end{enumerate}

\end{proposition}

\begin{proof}
 The proofs are in Propositions 5.1, 6.1 and Corollary 5.3 of \cite{Mar3}.
\end{proof}

\begin{definition}
 Let $A$ be a multiring. 
 
 \begin{itemize}
     \item A subset $T \subseteq A$ is called a preorder if $A^2 \subseteq T, T \cdot T \subseteq T$ and $T + T \subseteq T$ and it is proper if $-1 \notin T$.
     
     \item If $T$ is a proper preorder of $A$, the pair $(A,T)$ is called a p-multiring.
     
     \item We denote by $\mbox{sper}_T(A) = \{\sigma \in \mbox{sper}(A) \colon \sigma(T) \subseteq \{0,1\}\}$ the orders that preserve $T$.
 \end{itemize}
 
 If $(A,T)$ and $(B,P)$ are p-multirings, a multiring morphism $f \colon A \to B$ is a morphism of p-multirings if $f(T) \subseteq P$. We denote by $\textbf{pMult}$ the category of all p-multirings and its morphisms.
\end{definition}

Let $A$ be multiring and $T \subseteq A$ a proper preorder (note that $A$ is necessarily semi-real). Consider the natural map $\hat{ } \colon A \to 3^{sper_T(A)}$ given by $\hat{a} (\sigma) = \sigma (a)$. Let $Q_T(A) = \{\hat{a} \colon \mbox{for all } a \in A\}$.
This set has a product given by $\hat{a} \cdot \hat{b} = \hat{ab}$ and a sum given by $\hat{a} \in \hat{b} + \hat{c}$ if exists $a',b',c' \in A$ with $\hat{a'} = \hat{a}, 
\hat{b'} = \hat{b}, \hat{c'} = \hat{c}$ and $a' \in b' + c'$. If \textcolor{red}{$T = 1 + \sum A^2$} we denote $Q_T(A)$ just by $Q(A)$. 

\begin{fact} \label{Q(A)Mult}
 Let $(A,T)$ be a p-multiring and $\pi \colon A \to Q_T(A)$ the natural projection. Then $Q_T(A)$ is a multiring and given $a,b,c \in Q_T(A)$, $\pi(a) \in \pi(b) + \pi(c)$ if, and only if, for all $\sigma \in \mbox{sper}_T(A)$ we have $\sigma(a) \in \sigma(b) + \sigma(c)$. 
\end{fact}

\begin{proof}
 Proposition 7.3 of \cite{Mar3}.
\end{proof}

\begin{fact} \label{piIso}
 Let $A$ be a semi-real multiring. Then the map $A \to Q(A)$ is an isomorphism if, and only if, for all $a,b \in A$
 
\begin{enumerate}[i)]
 \item $a^3 = a$.
 \item $a + b^2a = \{a\}$
 \item $a^2 + b^2$ has an unique element.
\end{enumerate}

\end{fact}

\begin{proof}
 Proposition 7.5 of \cite{Mar3}.
\end{proof}

\begin{definition} \label{rrmdefinition}
 A multiring $A$ is called real reduced multiring (RRM) if it is semi-real and satisfies the three conditions of above theorem. The category of all real reduced multirings with multiring morphism is denoted by $\textbf{RRM}$.
\end{definition}

\begin{remark}
 \begin{itemize}
     \item 
If $(A,T)$ is a p-multiring, then by Fact $\ref{Q(A)Mult}$ the multring $Q_T(A)$ is a real reduced multiring.
     
     \item 
Let $f \colon (A,T) \to (B,P)$ a p-morphism. Then we have an induced map $\mbox{sper}_P(B) \to \mbox{sper}_T(A)$ given composition with $f$. 
 \end{itemize}
\end{remark}

\begin{corollary} \label{sepTheorem}
Let $A$ be a real reduced multiring and $a,b,c \in A$. Then $a \in b + c$ if, and only if, for all $\sigma \in \mbox{sper}(A)$ we have $\sigma(a) \in \sigma(b) + \sigma(c)$. In particular, if $\sigma(a) = \sigma(b)$ for all $\sigma \in \mbox{sper}(A)$, then $a = b$.
\end{corollary}

\begin{proof}
Direct consequence of Fact $\ref{Q(A)Mult}$ and Fact $\ref{piIso}$.
\end{proof}

\begin{remark}
Let $A$ be a multiring that satisfies the conditions of Fact $\ref{piIso}$. Given $a,b \in A$ we have $c \in a^2 + b^2$ satisfying $c^2 = c$ because since $a^2b^2 + a^2 = \{a^2\}$ and $a^2b^2 + b^2 = \{b^2\}$, $c^2 \in (a^2 + b^2) (a^2 + b^2) \subseteq a^2 + a^2b^2 + b^2 + a^2b^2 = a^2 + b^2$ and thus $c^2 = c$. Therefore $A$ semi-real if, and only if, $1 \neq 0$ (if $-1 \in \sum A^2$, then $-1 = x^2$ for some $x \in A$ and so $0 \in 1 + x^2 = \{1\}$). Furthermore, if $A$ is RRM, since $\sum A^2 = A^2 = \mbox{Id}(A)$, then $(A, \mbox{Id}(A))$ is a p-multiring.
\end{remark}

\begin{corollary} \label{hyperfieldRR}
 Let $F$ be a hyperfield. Then $F$ is a real reduced hyperfield if, and only if for all $a \in \dot{F}$
 
 \begin{enumerate}[.]
  \item $1 \neq 0$.
 
  \item $a^2 = 1$, whenever $a \neq 0$.
  
  \item $1 + 1 = \{1\}$.
 \end{enumerate}

\end{corollary}

\begin{proof}
 If $F$ is a real reduced hyperfield, then it is imediate that $F$ satisfies the above axioms. Reciprocally, $F$ is semi-real and $a^3 = a$ for all $a \in F$. Furthermore, given $x \in a + b^2 a$, if $b = 0$, then $x = a$ and if $b \neq 0$, then $x \in a + a = \{a\}$. It is also easy to see that $a^2 + b^2$ has just a unique element $0$ or $1$.
\end{proof}

\begin{theorem} \label{uniProMRR}

\begin{enumerate}[i)]
 \item Let $f \colon (A,T) \to (B,P)$ a morphism of p-multirings. Then exist unique map $Q(f) \colon Q_T(A) \to Q_P(B)$ such that
 
 \[\begin{tikzcd}
        (A,T) \arrow[r, "f"] \arrow[d, ""] & (B,P) \arrow[d, ""] \\
        Q_T(A) \arrow[r, "Q(f)"] & Q_P(B) \\
\end{tikzcd}\]
 
 is a commutative diagram. Furthermore, $Q$ defines a functor from category $\textbf{pMulti}$ to $\textbf{RRM}$.
 
 \item

 Let $(A,T)$ be a p-multiring and let $f \colon (A,T) \to (R, \sum R^2)$ be a morphism of multirings where $R$ is RRM. Then exist unique $\overline{f} \colon Q_T(A) \to R$ morphism such that
 the diagram

 \[\begin{tikzcd}
        A \arrow[r, ""] \arrow[dr, "f"] & Q_T(A) \arrow[d, "\overline{f}"] \\
         & R \\
\end{tikzcd}\]
 
 is commutative. In other words, the functor $Q \colon \textbf{pMulti} \to \textbf{RRM}$ is left-adjoint to the inclusion functor $i \colon \textbf{RRM} \to \textbf{pMulti}$ given by $i(R) = (R, \sum R^2)$.
 \end{enumerate}
 
\end{theorem}

\begin{proof}
 \begin{enumerate}[i)]
  \item Let $\pi_{(A,T)} \colon A \to Q_T(A)$ and $\pi_{(B,P)} \colon B \to Q_P(B)$ be the canonical projections. Consider the map $Q(f) \colon Q_T(A) \to Q_P(B)$ given by $Q(f) (\pi_{(A,T)}(a)) = \pi_{(B,P)}(f(a))$ (note that if $\pi_{(A,T)}(a) = \pi_{(A,T)}(b)$, then in particular for all $\sigma \in \mbox{sper}_P(B)$ we have $\sigma \circ f \in \mbox{sper}_T(A)$ and so $\sigma(f(a)) = \sigma(f(b))$; thus $\pi_{(B,P)}(f(a)) = \pi_{(B,P)} (f(b))$). So given $a,b,x \in A$ 
  
  \begin{itemize}
      \item $Q(f)(\pi_{(A,T)}(a) \pi_{(A,T)}(b)) = Q(f)(\pi_{(A,T)}(ab)) = \pi_{(B,P)} (f(ab)) = \pi_{(B,P)}(f(a)) \pi_{(B,P)}(f(b)) = \\ Q(f)(\pi_{(A,T)}(a))Q(f)(\pi_{(A,T)}(b))$.
      
      \item If $\pi_{(A,T)}(x) \in \pi_{(A,T)}(a) + \pi_{(A,T)}(b)$, then for all $\sigma \in \mbox{sper}_P(B)$ we have $\sigma \circ f \in \mbox{sper}_T(A)$ and so $\sigma (f (x)) \in \sigma ( f (a)) + \sigma ( f (b))$. Therefore, by Fact $\ref{Q(A)Mult}$, $\pi_{(B,P)}(f(x)) \in \pi_{(B,P)}(f(a)) + \pi_{(B,P)} (f(b))$, that is, $Q(f)(\pi_{(A,T)}(x)) \in Q(f)(\pi_{(A,T)}(a)) + Q(f)(\pi_{(A,T)}(b))$,
\end{itemize}
    Thus $Q(f) \colon Q_T(A) \to Q_P(B)$ is a multiring morphism -and unique that satisfies the diagram commutativity because $\pi_{(A,T)}, \pi_{(B,P)}$ are surjective.
    
    \item Consider the map $\overline{f} \colon Q_T(A) \to R$ given by $\overline{f} (\pi_{(A,T)}(a)) = f(a)$ -note that if $\pi_{(A,T)}(a) = \pi_{(A,T)}(b)$, then in particular for all $\sigma \in \mbox{sper}(R)$ we have $\sigma \circ f \in \mbox{sper}_T(A)$ and so $\sigma(f(a)) = \sigma(f(b))$; thus $f(a) = f(b)$ because $R$ is RRM). The proof that $\overline{f}$ is multiring morphism is analogous to the preceding item. Thus $\overline{f}$ is the unique morphism satisfying the diagram commutativity because $\pi$ is surjective.
 \end{enumerate} 
\end{proof}

\begin{corollary} \label{sper(Q(A))}
 Let $(A,T)$ be a p-multiring. Then the natural map $\pi_{(A,T)} \colon A \to Q_T(A)$ induces an homeomorphism $\mbox{sper}(Q_T(A)) \to \mbox{sper}_T(A)$.
\end{corollary}

\begin{proof}
 Direct consequence of Theorem $\ref{uniProMRR}$.
\end{proof}

\begin{lemma} \label{isoRS}
 Let $f \colon A \to B$ be surjective multiring morphism with $A,B$ RRM. If $\mbox{sper}(f) \colon \mbox{sper}(B) \to \mbox{sper}(A)$ is surjective, then $f$ is isomorphism and thus $\mbox{sper}(f)$ is an homeomorphism.
\end{lemma}

\begin{proof}
 First we will prove that $f$ is injective. Let $x,y \in A$ with $f(x) = f(y)$. Given $\sigma \in \mbox{sper}(A)$, by the surjective of $\mbox{sper}(f)$, exist $\tau \in \mbox{sper}(B)$
 such that $\sigma = \tau \circ f$. Thus $\sigma(x) = \tau(f(x)) = \tau(f(y)) = \sigma(y)$. Since $\sigma \in \mbox{sper}(A)$ was arbitrary, we have $x = y$. To conclude that $f$ is an isomorphism, we need to show that if $f(x) \in f(y) + f(z)$ for $x,y,z \in A$, then $x \in y + z$. But using the same argument for injectivity we have $\sigma (x) \in \sigma(y) + \sigma(z)$
 for all $\sigma \in \mbox{sper}(A)$ and so $x \in y + z$.
\end{proof}

\begin{theorem} \label{1sumSqua}
 Let $A$ be a semi-real multiring and $\pi \colon A \to Q(A)$ the canonical projection.
 
 \begin{enumerate}[i)]
  \item The multiring $A /_m 1 + \sum A^2$ is semi-real and the canonical map $A \to A /_m 1 + \sum A^2$ induces an isomorphism $Q(A) \cong Q(A /_m 1 + \sum A^2)$.
  
  \item If $A$ is hyperfield, then the morphism $\pi$ induces isomorphisms $A /_m \sum \dot{A}^2 \cong Q(A) \cong A /_m 1 + \sum A^2$.
 \end{enumerate}

\end{theorem}

\begin{proof}

\begin{enumerate}[i)]
\item Given a morphism $\sigma \colon A \to 3$, $\sigma (1 + \sum A^2) = \{1\}$. Thus by Proposition $\ref{marshallQ}$, the canonical morphism $\rho \colon A \to A /_m 1 + \sum A^2$ induces a bijection $\mbox{sper}(A /_m 1 + \sum A^2) \cong \mbox{sper}(A)$ and then by Proposition $\ref{basicReal}$ $A$ is semi-real if, and only if, $A/1 + \sum A^2$ is semi-real. 

Furthermore, since $\rho$ is surjective, $Q(\rho) \colon Q(A) \to Q(B)$ is surjective and by Corollary $\ref{sper(Q(A))}$ $\mbox{sper}(Q(\rho)) \colon \mbox{sper}(Q(B)) \to \mbox{sper}(Q(A))$ is surjective (in fact an homeomorphism). Then by Lemma $\ref{isoRS}$ the morphism $Q(\rho)$ is an isomorphism.

  \item By item $i)$, we have $Q(F) \cong Q(F/_m 1 + \sum F^2)$. On the other hand,
  
  \begin{itemize}
  \item Since $F$ is semi-real, $1 \neq 0$ in $F/_m 1 + \sum F^2$.
  \item If $x \neq 0$ in $F$, then $x^2 (1 + x^{-2}) = (1 + x^2)$ and so $x^2 = 1$ in $F/_m 1 + \sum F^2$.
  \item If $y \in 1 + 1$ in $F/_m 1 + \sum F^2$, then exists $s_1, s_2, s_3 \in 1 + \sum F^2$ such that $ys_1 \in s_2 + s_3$. Since $1 + \sum F^2$ is closed under sums, we have
  $ys_1 \in 1 + \sum F^2$. Thus $y = 1$ in $F/_m 1 + \sum F^2$.
  \end{itemize}

  Therefore by Corollary $\ref{hyperfieldRR}$, the hyperfield $F/_m 1 + \sum F^2$ is real reduced. Then follows by Theorem $\ref{uniProMRR}$ that 
  $Q(F/_m 1 + \sum F^2) \cong F/_m 1+\sum F^2$. On the other side, since $1 + \sum F^2 \subseteq \sum \dot{F}^2$, we have a surjective morphism 
  $f \colon F /_m 1 + \sum F^2 \to F /_m \sum \dot{F}^2$.
  
  \begin{claim*}
   Let $x,y,z \in F$ such that $x \in y + z$ in $F /_m \sum \dot{F}^2$. Then $x \in y + z$ in $F /_m 1 + \sum F^2$.
  \end{claim*}
  
  \begin{proof}
  By hypothesis, exists $s_1, s_2, s_3 \in \sum \dot{F}^2$ such that $xs_1 \in ys_2 + zs_3 \ (*)$. Thus it should exist $x_1, x_2,x_3 \in \dot{F}$ such that
  $s_ix_i^{-2} \in 1 + \sum F^2$, $i=1,2,3$ (if $s \in x_1^2 + \cdots + x_n^2$ and $x_1 \neq 0$, then $sx_1^{-2} \in 1 + \sum F^2$). But by item $ii)$ above exists $m_i,n_i \in 1 + \sum F^2$ such that $m_ix_i^{-2} = n_i$.
  Consequently, multiplying $(*)$ by $n_1n_2n_3$ 
  
  \begin{align*}
   x (s_1x_1^{-2})m_1n_2n_3 \in y(s_2x_2^{-2})m_2n_1n_3 + z(s_3x_3^{-2})m_3n_1n_2.
  \end{align*}

  Therefore $x \in y + z$ in $F/_m1 + \sum F^2$.
  \end{proof}

  By the above claim, the morphism $f$ reflects the multivaluated sum. Thus to conclude that $f$ is an isomorfism, we have to prove that $f$ is injective but if $f(x) = f(y)$, then $f(x) \in f(y) + f(0)$ in $F/_m \sum \dot{F}^2$. So by the claim $x \in y + 0$ in $F/_m 1 + \sum F^2$, that is, $x = y$.
  
  Thus $Q(F) \cong F/_m 1 + \sum F^2 \cong F/_m \sum \dot{F}^2$.

  \end{enumerate}
\end{proof}


\section{Structural presheaf for Multirings}

Here we present a structural presheaf for multiring. This presheaf generalize the sheaf for rings, the sheaf for hyperdomains in \cite{Jun} and the sheaf for real reduced multrings in \cite{DP2} (in the language of real semigroups). Thus rings, hyperdomains and real reduced multirings are geometric (the presheaf is a sheaf). But there are multirings non-geometric (example $\ref{nonGeometric}$) and in the next section it is given a first-order characterization of geometric von Neumann hyperrings. \\

Let $A$ be a multiring, $\alpha$ an ideal and $a \in A$. We define $\sqrt{\alpha} = \{x \in A \colon \mbox{ exist natural } n \geq 1 \mbox{ such that } x^n \in \alpha\}$
and $S_a =\{x \in A \colon \mbox{ exists } n \geq 0 \mbox{ such that } a^n \in (x)\} $ (if $a = 0$, we consider $0^0  = 1$.). We also define $A_a = \{1, a, a^2, \ldots\}^{-1}A$. Note that, if $\rho \colon A \to A_a$ is the canonical map, then $S_a = \rho^{-1} (A_a^{\times_w})$. 

\begin{proposition}\label{sheaf2}
 Let $A$ multiring. Let $\alpha$ be an ideal and $a \in A$.
 
 \begin{enumerate}[i)]
  \item $\sqrt{\alpha}$ is an ideal and $S_a$ is a multiplicative set.
 
  \item $\sqrt{\alpha} = \bigcap_{\alpha \subseteq p \in spec(A)} p$. In particular, $\bigcap_{p \in spec(A)} p$ is the set of nilpotents.
  
  \item $S_a = \bigcap_{a \notin p}p^c$. In particular, $\bigcap_{p \in spec(A)} p^c = A^{\times_w}$.
  
  \item  Furthermore, the following are equivalent:
  
  \begin{enumerate}[a)]
   \item $D(a) \subseteq D(b)$.
   
   \item $\sqrt{a} \subseteq \sqrt{b}$.
   
   \item $a \in \sqrt{b}$.
   
   \item $S_b \subseteq S_a$.
  \end{enumerate}

  \item If $D(a) \subseteq D(b)$, then exist unique morphism $\rho_{D(b),D(a)} \colon S_b^{-1}A \to S_a^{-1}A$ such that
  
  \[
  \begin{tikzcd}
    & A \arrow{dr} \arrow{dl} \\ 
  S_b^{-1}A \arrow{rr}{\rho_{D(b), D(a)}} && S_a^{-1} A
  \end{tikzcd}
  \]
  
  is a commutative diagram. In addition, if $D(a) \subseteq D(b) \subseteq D(c)$, then $\rho_{D(c),D(a)} = \rho_{D(c),D(b)} \circ \rho_{D(c),D(b)}$ and $\rho_{D(a), D(a)} = Id_{S_a^{-1}A}$.
  
  \item Exists unique morphism $ A_a \to S_a^{-1}A$ such that
    
    \[
  \begin{tikzcd}
    & A \arrow{dr} \arrow{dl} \\ 
  A_a \arrow{rr} && S_a^{-1} A
  \end{tikzcd}
  \]
  is a commutative diagram.

 \end{enumerate}

\end{proposition}

\begin{proof}
 
 \begin{enumerate}[i)]
  \item Given $x, y \in \sqrt{\alpha}$, choose $n \geq 0$ such that $x^n, y^n \in \alpha$. Then given $z \in x - y$, we have $z^{2n} \in \sum_{i=0}^{2n} {2n \choose i} x^i (-y)^{2n-i} \subseteq \alpha$.
  Since for all $i=0, \ldots, 2n$, $i \geq n$ or $2n-i \geq n$, we have $z \in \sqrt{\alpha}$. Then it is easy to see that $\sqrt{\alpha}$ is an ideal.
  
  Note that $1 \in S_a$ and given $x,y \in S_a$, exists $m,n \geq 0$ and $t_1, \ldots, t_n, l_1, \ldots, k_k \in A$ such that $a^m \in xt_1 + \cdots + xt_n$ and $a^n \in yl_1 + \cdots + yl_k$. Then $a^{m+n} \in \sum_{i,j}xyt_il_jl$ and so $xy \in S_a$.
  
  \item It is easy to see that $\sqrt{\alpha} \subseteq \bigcap_{\alpha \subseteq p} p$. Let $x \notin \sqrt{\alpha}$ and consider the multiplicative set $S = \{1, x, x^2,\ldots\}$. Since $\alpha \cap S = \emptyset$, by Theorem $\ref{PIT}$ exist a prime ideal $p$ with $\alpha \subseteq p$ and $x \notin p$. Then $x \notin \bigcap_{\alpha \subseteq p} p$.
  
  \item It is imediate that $S_a \subseteq \bigcap_{a \notin p} p^c$. Reciprocally, let $x \notin S_a$. Then $(x) \cap \{1, a, a^2, \ldots\} = \emptyset$ and by Theorem $\ref{PIT}$ exist a prime ideal $p$ such that $x \in p$ and $a \notin p$. Thus $x \notin \bigcap_{a \notin p} p^c$.
  
  \item Let $\alpha = (a_1) + \cdots + (a_n)$. Given any prime ideal $p$, we have $\alpha \not \subseteq p$ if, and only if, exists $1 \leq i \leq n$ such that $a_i \notin p$. Thus by item $ii)$ $\sqrt{\alpha} = \bigcap_{\alpha \subseteq p} p$ and $\sqrt{(a)} = \bigcap_{a \in p} p$. Therefore
  \begin{align*}
      D(a) \subseteq D(a_1) \cup \cdots \cup D(a_n) & \Leftrightarrow \mbox{ for all } p \in \mbox{spec}(A),\ a \notin p \rightarrow \alpha \not \subseteq p \\
      & \Leftrightarrow \sqrt{(a)} \subseteq \sqrt{\alpha}.
  \end{align*}

  In particular, we have the equivalence $a) \Leftrightarrow b$ and by $iii)$ $b) \Leftrightarrow d)$. We also have $b) \Rightarrow c)$.
 
 $c) \Rightarrow b)$: Since $a \in \sqrt{b}$, exists $n \geq 0$ such that $a^n \in (b) \ (*)$. Let $x \in \sqrt{a}$. Then exists $m \geq 0$ and $l_1, \ldots, l_o \in A$ such that $x^m \in al_1 + \cdots + al_o$. Thus we have $x^{nm} \in \sum_{i_1 + \cdots + i_o = n} {n \choose i_1, \ldots, i_o} a^n l_1^{i_1} \cdots l_o^{i_o} \subseteq (b)$ by $(*)$.
 
  \item If $D(a) \subseteq D(b)$, we have $S_b \subseteq S_a$ and so the existence and uniqueness of $\rho_{D(b),D(a)}$ follows by $\ref{localization}$. The other stated properties follows
  by the uniqueness just proved.
  
  \item Since $\{1, a, a^2, \ldots\} \subseteq S_a$, the existence and uniqueness of the morphism follows by $\ref{localization}$. 
 \end{enumerate}

\end{proof}

Let $A$ be a multiring and $a \in A$. We say that $A$ has the $a$-invertible property if $A_a^{\times_w} = A_a^{\times}$ and $A$ has the invertible property if it has the $a$-invertible property for every $a \in A$.

\begin{proposition}
 Let $A$ be a multiring and $a \in A$. Let $\rho_a \colon A \to A_a$, $\rho_{S_a} \colon A \to S_a^{-1}A$ be the canonical morphisms. Let $\overline{\rho_{S_a}} \colon A_a \to S_a^{-1}A$ the induced morphism.  The following are equivalent:
 
 \begin{enumerate}[i)]
    \item $\rho_a(S_a) \subseteq A_a^{\times}$.
    \item $A$ satisfies the $a$-invertible property.
    \item $S_a = \{x \in A \colon \mbox{ exists } n \geq 0,y \in A \mbox{ such that } a^n = xy\}$.
    \item $\overline{\rho_{S_a}} \colon A_a \to S_a^{-1}A$ is an isomorphism.
 \end{enumerate}
 
  In particular, if for every $a \in A$, $(a) = Aa$, then $A$ has the invertible property.
\end{proposition}

\begin{proof}
$i) \Rightarrow ii)$: Let $\frac{x}{a^k} \in A_a^{\times_w}$. By definition, exists $\frac{t_1}{a^{n_1}}, \ldots, \frac{t_l}{a^{n_l}} \in A_a$ such that $1 \in \frac{x}{a^k} \frac{t_1}{a^{n_1}} + \cdots + \frac{x}{a^k}\frac{t_1}{a^{n_1}}$. Thus exists $m, n_1', \ldots, n_l'$ such that $a^m \in xt_1a^{n_1'} + \cdots + xt_la^{n_l'}$ in $A$ and so $x \in S_a$. By hypothesis, $\rho(x) \in A_a^{\times}$ and exists $\frac{y}{a^t} \in A_a$ with $\frac{x}{1} \frac{y}{a^t} = 1$. Then $\frac{x}{a^k} \in A_a^{\times}$.

$ii) \Rightarrow iii)$: Let $x \in S_a$. Then exists $n \geq 0$ with $a^n \in (x)$ and so $x \in A_a^{\times_w} = A_a^{\times}$. By definition, this means that exists $\frac{y}{a^t} \in A_a$ with $\frac{x}{1} \frac{y}{a^t} = 1$. Then exists $m \geq 0$ such that $xya^m = a^{m+t}$.

$iii) \Rightarrow iv)$: Observe that, given $x \in S_a$, by hypothesis exists $n \geq 0 $ and $y \in A$ with $a^n = xy$ and thus $\rho_a(x) \in A_a^{\times}$. Thus by the universal property of $\rho_{S_a}$, exists morphism $\overline{\rho_a} \colon S_a^{-1}A \to A_a$ such that $\rho_a = \overline{\rho_a} \circ \rho_{S_a}$. To conclude, note that $(\overline{\rho_a} \circ \overline{\rho_{S_a}}) \circ \rho_a = \rho_a$ and $(\overline{\rho_{S_a}} \circ \overline{\rho_a}) \circ \rho_{S_a} = \rho_{S_a}$. Thus, by the uniqueness of universal property of $\rho_a$ and $\rho_{S_a}$, we have $\overline{\rho_a} \circ \overline{\rho_{S_a}} = Id_{A_a}$ and $\overline{\rho_{S_a}} \circ \overline{\rho_a} = Id_{S_a}$.

$iv) \Rightarrow i)$: Let $x \in S_a$. Since $\overline{\rho_{S_a}} (\rho_a(x)) = \rho_{S_a} (x) \in (S_a^{-1}A)^{\times}$ and $\overline{\rho_{S_a}}$ is an isomorphism, $\rho_a(x) \in A_a^{\times}$.

\end{proof}

\begin{example}
\begin{itemize}
    \item If $A$ is an hyperring, then for every $a \in A$ $(a) = aA$; thus $A$ has the invertible property.
    
    \item Let $A$ be a real reduced multiring and $a,t_1, \ldots, t_n \in A$. Let $x \in at_1 + \cdots + at_n$. Then given $\sigma \in \mbox{sper}(A)$, if $\sigma(a) = 0$, we have $\sigma(x) = 0$; thus $\sigma(x) = \sigma(a^2x)$. Since $\sigma \in \mbox{sper}(A)$ was arbitrary, by $\ref{sepTheorem}$ we have $x = a^2x = a(ax)$. Therefore $(a) = aA$ and $A$ has the invertible property.
    
    \item Let $A$ be a multiring and consider $S_1 = \{x \in A \colon 1 \in (x)\} = A^{\times_w}$. Then $S_1^{\times}A$ has the 1-invertible property. In fact, if $\frac{x}{s} \in S_1^{\times}A$ is a weak-invertible, then $1 \in (\frac{x}{s})$. This implies that exists $u \in S_1$ with $u \in (x)$ in $A$ but since $1 \in (u)$ we have $1 \in (x)$. Thus $x \in S_1$ and $\frac{x}{s} \in S_1^{\times}A$ is invertible. Furthermore, the natural map $A \to S_1^{\times}A$ is initial between morphism from $A$ to multirings with the 1-invertible property.
\end{itemize}
\end{example}

Let $A$ be a multiring and $\mathcal{B}_A \coloneqq \{D(a)\}_{a \in A}$ be the set of basic open sets of $\mbox{spec}(A)$. Consider the relation $\mathcal{F}_A \colon \mathcal{B}_A \to \textbf{MultR}$ given by
$\mathcal{F}_A (D(a)) = S_a^{-1}A$ and if $D(a) \subseteq D(b)$, the morphism $\mathcal{F}_{A, D(b), D(a)} = \rho_{D(b), D(a)}$. By $\ref{sheaf2}$, $\mathcal{F}_A$ is in fact a functor. When there is no risk of confusion, $\mathcal{F}_A$ is denoted just by $\mathcal{F}$ and for $D(a) \subseteq D(b)$, the restriction map
$\mathcal{F}_{D(b), D(a)}$ is denoted by $\mathcal{F}_{b,a}$. Furthermore, given $p \in \mbox{spec}(A)$, we denote by $\mathcal{F}_p = \varinjlim_{a \notin p} \mathcal{F}(D(a))$ the stalk at $p$.

\begin{definition}
 Let $A$ be a multiring. The functor $\mathcal{F}_A$ is called structural presheaf associated to $A$. 
 
 \begin{enumerate}[i)]
  \item $\mathcal{F}$ is a monopresheaf if for all cover $D(e) = \bigcup_{i \in I} D(e_i)$ and all $a,b,c \in A_{e}$ with 
  
   \begin{align*}
   \mathcal{F}_{e, e_i} (a) \in \mathcal{F}_{e, e_i} (b) + \mathcal{F}_{e, e_i}(c) \mbox{ in } A_{e_i} \mbox{ for every } i \in I,
   \end{align*}

   then $a \in b + c$ in $A_e$.
   
   \item $\mathcal{F}$ is a sheaf if it is a monopresheaf and if for all cover $D(e) = \bigcup_{i \in I} D(e_i)$ and for $x_i \in A_{e_i}$ for all $i \in I$ with
 
   \begin{align*}
    \mathcal{F}_{e_i, e_ie_j} (x_i) = \mathcal{F}_{e_j, e_ie_j} (x_j) \mbox{ for every } i,j \in I,
   \end{align*}

  then exist (necessarily unique) $x \in A_e$ such that $\mathcal{F}_{e, e_i} (x) = x_i$ for every $i \in I$.
 \end{enumerate}

\end{definition}

\begin{definition}
Let $A$ be a multiring. If $\mathcal{F}_A$ is a monopresheaf, $A$ is called a mono-multiring; if $\mathcal{F}_A$ if a sheaf, then $A$ is called geometric.
\end{definition}

\begin{proposition} \label{fiber}
 Let $A$ a multiring. Given $p \in \mbox{spec}(A)$, then $\mathcal{F}_p \cong A_p$ naturally. 
\end{proposition}

\begin{proof}
  Given $a \notin p$, since $S_a \subseteq A \setminus p$, exist unique map $i_{S_a} \colon S_a^{-1} A \to A_p$ such that
 
    \[
  \begin{tikzcd}
    & A \arrow{dr}{\rho_p} \arrow{dl} \\ 
  S_a^{-1}A \arrow{rr}{i_{S_a}} &&  A_p
  \end{tikzcd}
  \]
  
  is a commutative diagram. Thus, by the uniqueness property, it is easy to see that if $D(a) \subseteq D(b)$, then
  
    \[
  \begin{tikzcd}
    & A_p  \\ 
  S_b^{-1}A \arrow{rr}{\rho_{D(b),D(a)}} \arrow{ur}{i_{S_b}} &&  S_a^{-1}A \arrow{ul} [above] {i_{S_a}}
  \end{tikzcd}
  \]  
 
 is a commutative diagram. Thus, the family $\mathcal{I} = \{i_{S_a} \colon S_a^{-1}A \to A_p \colon a \notin p\}$ is a cocone. To see it is in fact a limit cocone,
 consider a family of morphism $\{h_{S_a} \colon S_a^{-1}A \to B \colon a \notin p\}$ such that if $D(a) \subseteq D(b)$, $h_{S_b} = h_{S_a} \circ \rho_{D(b),D(a)}$.
 We have to prove that exist unique morphism $h \colon A_p \to B$ such that given $a \notin p$, $h_{S_a} = h \circ i_{S_a}$.

 \textit{Existence:} Define $h' = h_{S_1} \circ \rho_{S_1} \colon A \to B$.
 To use the universal property of the map $\rho_p \colon A \to A_p$, we have to prove that for all $a \notin p$, $h'(a) \in B^{\times}$. But given $a \notin p$, since

    \[
  \begin{tikzcd}
    & A \arrow{dr}{h'} \arrow{dl}[above]{\rho_{S_a}} \\ 
  S_a^{-1}A \arrow{rr}{h_{S_a}} &&  B
  \end{tikzcd}
  \]
 
 is a commutative diagram and $\rho_{S_a} (a) \in S_a^{-1}A$ is invertible, we have $h' (a) \in B^{\times}$. Then exist a morphism $h \colon A_p \to B$ such that
 
 \[
  \begin{tikzcd}
    & A \arrow{dr}{h'} \arrow{dl}[above]{\rho_p} \\ 
  A_p \arrow{rr}{h} &&  B
  \end{tikzcd}
  \]
 
 Then for all $a \notin p$, since $i_{S_a} \circ \rho_{S_a} = \rho_p$, we have that $h \circ i_{S_a} \circ \rho_{S_a} = h \circ \rho_p = h' = h_{S_a} \circ \rho_{S_a}$.
 But since $\rho_{S_a}$ is an epimorphism, we have $h \circ i_{S_a} =  h_{S_a}$. \\
 \textit{Uniqueness:} Just note that for every $x \in A_p$, exists $a \notin p$ and $y \in S_a^{-1}A$ such that $i_{S_a} (y) = x$.

\end{proof}

\begin{remark} \label{fibertoOpen}
 Let $A$ be a multiring, $a,b,c \in A$ and $p \in \mbox{spec}(A)$. Note that if $a \in b + c$ in $A_p$, then exists $x \notin p$ with $ax \in bx + cx$. Thus $a \in b + c$ in $A_x$, with $p \in D(x)$. Thus given $e \in A$, if $a \in b + c$ in $A_p$ for all $p \in D(e)$, then exists $e_1, \ldots, e_n$ with $D(e) = \bigcup_{i=1}^n D(e_i)$ and $a \in b + c$ in $A_{e_i}$ for all $i=1, \ldots, n$ because $D(e)$ is compact.
\end{remark}

\section{Geometric von Neumann Regular Hyperrings}

\begin{proposition} \label{cVn}
 Let $A$ be an hyperring. The following are equivalent:
 
 \begin{enumerate}[i)]
  \item $\mbox{spec}(A)$ is a Boolean topological space and $\sqrt{0} = 0$.
  
  \item For all $a \in A$ exist $b \in A$ such that $a = a^2b$.
 \end{enumerate}

\end{proposition}


\begin{proof}
  $i) \Rightarrow ii)$ Let $a \in A$. Since $\mbox{spec}(A)$ is a boolean topological space, exists $a_1, \ldots, a_n \in A$ such that
  $D(a)^c = D(a_1) \cup \cdots \cup D(a_n)$. In particular, $D(aa_i) = D(a) \cap D(a_i)  = \emptyset$ and so $aa_i \in \sqrt{0}=0$ for all $i=1, \ldots, n$.
  Seeing that $\mbox{spec}(A) = D(a) \cup D(a)^c = D(a) \cup D(a_1) \cup \cdots \cup D(a_n)$, by Proposition $\ref{sheaf2}$ exists $b, t_1, \ldots, t_n \in A$ such that $1 \in ab + a_1t_1 + \cdots +a_nt_n$. Then multiplying by $a$ entails $a = a^2b$.
  
  $ii) \Rightarrow i)$ Let $a \in A$ and take $b$ with $a = a^2b$. Then, since $A$ is an hyperring, exist $l \in 1-ab$ such that
  $al=0$. It is easy to see that $D(a)^c = D(l)$ and thus $\mbox{spec}(A)$ is a Boolean topological space. Furthermore, given $x \in \sqrt{0}$, exists $n \geq 1$ such that $x^n = 0$. Let $b \in A$ with $x = x^2b$. Since $(xb)^2 = x^2b^2 = xb $, we have $xb = (xb)^n = x^n b^n = 0$ and so $x = (xb)b = 0$.

\end{proof}

\begin{definition}
 Let $A$ be an hyperring. If any of the equivalent conditions of $\ref{cVn}$ is valid, $A$ is called von Neumann hyperring. The category
 of all von Neumann hyperring with the usual notion of morphism is denoted by $HVN$.
\end{definition}

\begin{remarknot} \label{vNremarks}
\begin{itemize}
Let $A$ be a von Neumann hyperring.

\item Since all axioms of von Neumann hyperring are geometric sentences, the inductive limit of vNH are again vNH.

\item Since $\mbox{spec}(A)$ is Boolean, by remark $\ref{remarkOpen}$ every prime ideal of $A$ is maximal.

 \item The basic open sets of $\mbox{spec}(A)$ are given by idempotents elements. Indeed,
 given $a \in A$, take $b \in A$ with $a = a^2b$; then $ab$ is idempotent and $D(a) = D(ab)$. 
 
 \item For each $a \in A$ idempotent, since $a^2 = a$ and $A$ is an hyperring, exist $a^c \in 1 - a$ such that $a a^c = 0$. First note that if $x,y\in 1 - a$ and $ax = ay = 0$, then
 $xy \in y - ya = \{y\} $ and $xy \in x - xa = \{x\}$. So $x = xy = y$ and $a^c$ is unique determined. Furtheremore, $(a^c)^2 \in a^c - aa^c = \{a^c\}$ and thus
 $a^c$ is idempotent. We also have $D(a)^c = D(a^c)$.
 
 \item Given $a,b \in A$ idempotent, if $D(a) = D(b)$, then $D(ab^c) = D(a) \cap D(b)^c = \emptyset = D(a^cb)$. Thus, since $\sqrt{0} = 0$,
 $0 = ab^c \in a(1-b)$ and $0 = a^cb \in b(1-a)$. Then $a = ab = b$.
 
 \item Given $a \in A$, we denote by $i(a)$ the (unique) idempotent with $D(a) = D(i(a))$ and define $a^c \defeq i(a)^c$.
 
\item The topological space $\mbox{sper}(A)$ is also Boolean. Given $a \in A$, let $x \in a - a^c$. Let $\sigma \in \mbox{sper}(A)$. If $\sigma(a) = 0 = \sigma(i(a))$, then $\sigma(x) = -\sigma(i(a))^c = -1$; if $\sigma(a) \neq 0$, then $\sigma(a^c) = 0$ and $\sigma(x) = \sigma$. Thus $\sigma(x) \in \{1,-1\}$ and $\sigma(a) = 1$ if, and only if, $\sigma(x) = 1$. Therefore $U(a)^c = U(x)^c = U(-x)$.
 
 \item Let $f \colon A \to B$ be a morphism of von Neumann hyperring. Given $a \in A$, we have $f(i(a)) = i(f(a))$ and $f(a^c) = f(a)^c$. In fact, let $b \in A$ with $a = a^2b$. Then $i(a) = ab$ and $i(f(a)) = f(a)f(b)$ because $f(a) = f(a)^2f(b)$. Thus $f(i(a)) = i(f(a))$. On the other hand,
 since $i(a)^c \in 1 - i(a), i(a) i(a)^c = 0$, we have $f(i(a)^c) \in 1 - f(i(a)), f(i(a))f(i(a)^c) = 0$ and so $ f(i(a))^c = f(i(a)^c)$. Thus $f(a^c) = f(i(a)^c) = f(i(a))^c = i(f(a))^c = f(a)^c$.
 
 \item If $D(a) \subseteq D(b)$, $a,b$ idempotents, then $a \in \sqrt{(b)}$. Thus exist $l$ such that $a = bl$. Then $ab = bbl = a$.
 
\end{itemize}
\end{remarknot}

\begin{definition}
 Let $A$ be a von Neumann hyperring. A set $\{e_1, \ldots, e_n\} \subseteq \mbox{Id}(A)$ is partition if for all $i \neq j$, $e_ie_j = 0$ and it is a partition of unity if it is a partition and $1 \in e_1 + \cdots + e_n$.
\end{definition}

\begin{lemma} \label{orth}
Let $A$ be a von Neumann hyperring and $B = \{e_1, \ldots, e_n\}$ a partition.

\begin{enumerate}[i)]
    \item If $x \in e_1 + \cdots + e_n$, then $D(x) = \bigcup_{i=1}^n D(e_i)$. In particular, given $U \subseteq \mbox{spec}(A)$ clopen subset, exists $e \in Id(A)$ such that $U = D(e)$.
    
    \item Assume that $B$ is partition of unity and $e_i + e_i^c = \{1\}$ for all $i=1, \ldots, n$. Then $e_1 + \cdots + e_n = \{1\}$.
    
    \item If $B = \{e_1, \ldots, e_n\}$ and $\{f_1, \ldots, f_k\}$ are partitions of unity, then $\{e_1f_1, \ldots, e_1f_k, \ldots, e_nf_k\}$ is also a partition of unity and $e_1 + \cdots + e_n, f_1 + \cdots + f_k \subseteq e_1f_1 + \cdots + e_1f_k + e_2f_1 + \cdots + e_nf_k$.
\end{enumerate}

\end{lemma}

\begin{proof}
\begin{enumerate}[i)]
    \item If $p \in D(x)$, then exists $i$ with $p \in D(e_i)$. Reciprocally, let $p \in D(e_i)$. For all $j \neq i$ we have $e_je_i = 0 \in p$ and thus $e_j \in p$. Then if $x \in p$, then $e_i \in p$, an absurd. Thus $p \in D(x)$.
    
    Let $U \subseteq \mbox{spec}(A)$ clopen. Then exists $e_1, \ldots, e_n \in \mbox{Id}(A)$ with $U = \bigcup_{i=1}^n D(e_i)$. Let $e_1' = e_1, e_i' = e_i (e_1^c \cdots e_{i-1}^c)$, $i=2, \ldots, n$. Note that $D(e_i') = D(e_i) \cap \bigcap_{k=1}^{i-1} D(e_k)^c$ and $e_i'e_j' = 0$ if $i \neq j$. Then given $x \in e_1' + \cdots + e_n'$, we have $D(x) = \bigcup_{i=1}^n D(e_i') = \bigcup_{i=1}^n D(e_i) = U$.
    
    \item Given $i \neq j$, $e_i e_j = 0$ implies $e_ie_j^c = e_i$ because $e_j^c \in 1 - e_j$. Then $e_i e_1^c \cdots e_{i-1}^c = e_i$ for all $i =2, \ldots, n-1$ and multiplying $1 \in e_1 + \cdots + e_n$ by $e_1^c \cdots e_{n-1}^c$ entails $e_n = e_1^c \cdots e_{n-1}^c$. Therefore
    
    \begin{align*}
     e_1 + e_2 + \cdots + e_{n-1} + e_n & = e_1 + e_2 + \cdots + e_{n-1} e_1^c \cdots e_{n-2}^c + e_1^c \cdots e_{n-1}^c \\
							& =  e_1 + e_2 + \cdots + e_1^c \cdots e_{n-2}^c(e_{n -1} + e_{n-1}^c) \\
							& =  e_1 + e_2 + \cdots + e_{n-2}e_1^c \cdots e_{n-3}^c + e_1^c \cdots e_{n-2}^c \\
							& \vdots \\
							& = e_1 + e_1^c = \{1\}.
    \end{align*}
    
    \item Since $1 \in e_1 + \cdots + e_n$ and $1 \in f_1 + \cdots + f_k$, we have $1 = 1 \cdot 1 \in e_1f_1 + \cdots + e_1f_k + e_2f_1 + \cdots e_nf_k$ and if $(i,j) \neq (i', j')$, then $(e_if_j)(e_{i'}f_{j'}) = 0$. Furthermore, the relation $1 \in e_1 + \cdots + e_n$ implies that $f_j \in e_1f_j + \cdots + e_nf_j$ and so $f_1 + \cdots + f_k \subseteq e_1f_1 + \cdots + e_1f_k + e_2f_k + \cdots + e_nf_k$. By a symmetric same argument, we conclude $e_1 + \cdots + e_n \subseteq e_1f_1 + \cdots + e_1f_k + e_2f_1 + \cdots + e_nf_k$.
\end{enumerate}
\end{proof}

Let $A$ be a von Neumann hyperring. By the above lemma, $\mathcal{B}_A$ is the Boolean algebra of clopens sets of $\mbox{spec}(A)$. Let $\mathcal{F}$ be the structural pre-sheaf of $A$. Given $p \in \mbox{spec}(A)$, note that $pA_p \subseteq A_p$ is is the zero ideal because given $x \in p$, $x^c \notin p$ and $xx^c = 0$. Then by Proposition $\ref{kAp}$ the fiber $\mathcal{F}_p \cong A_p \cong A_p / pA_p \cong K_A(p)$ is an hyperfield.

\begin{theorem} \label{geoVon}
 Let $A$ be a von Neumann hyperring. The following are equivalent:
 
 \begin{enumerate}[i)]
  \item $A$ is a mono-hyperring.
  
  \item If $u \in A$ satisfies $u = 1$ in $K_A(p)$ for every $p \in \mbox{spec}(A)$, then $u=1$ em $A$.
  
  \item For every $e \in A$ idempotent, $e + e^c = \{1\}$.
 \end{enumerate}

 If any of the above conditions is valid, $A$ is geometric and $a \in b+c$ in $A$ if, and only if, $a \in b+c$ in 
 $K_A(p)$ for every $p \in \mbox{spec}(A)$. In particular, $a = b$ if, and only if, for all $p \in \mbox{spec}(A)$, $a = b$ in $K_A(p)$.
\end{theorem}

\begin{proof}
  $i) \Rightarrow ii)$: Let $u \in A$ with $u=1$ in $K_A(p)$ for all $p \in \mbox{spec}(A)$. By remark $\ref{fibertoOpen}$, exists $e_1, \ldots, e_n \in A$ with $\mbox{spec}(A) = \bigcup_{i=1}^n D(e_i)$ and $u = 1$ in $A_{e_i}$. Since $\mathcal{F}_A$ is a monopresheaf, it follows $u=1$.
  
  $ii) \Rightarrow iii)$: Immediate since any $u \in e + e^c$ is locally equal to $1$.
  
  $iii) \Rightarrow i)$:
    Assume that $\mbox{spec}(A) = \bigcup_{i=1}^n D(t_i)$ and $a \in b + c$ in $A_{t_i}$ for all $i=1, \ldots, n$. Consider the set $X = \{x \in A \colon ax \in bx + cx\}$.
    Then $AX \subseteq X$ and $I = \bigcup \{x_1 + \cdots + x_k \colon k \geq 1, x_i \in X \mbox{ for all } i=1, \ldots, k\}$ is the ideal generated by $X$. By hypothesis, $I$ is not contained in any
    prime ideal. Then $1 \in I$ and so exist $x_1, \ldots, x_k \in X$ such that $1 \in x_1 + \cdots + x_n$. Since exist $b_i \in A$ such that $b_ix_i$ is idempotents and $D(x_i) = D(b_ix_i)$,
    we can assume that exists $e_1, \ldots, e_k \in A$ idempotent such that
    
    \begin{enumerate}[.]
     \item $\mbox{spec}(A) = \bigcup_{i=1}^k D(e_i)$.
     \item $ae_i \in be_i + ce_i$ for all $i=1, \ldots, k$.
    \end{enumerate}

    Furtheremore, change $e_i$  by $e_i e_1^c \cdots e_{i-1}^c$, we can assume that $e_ie_j = 0$ if $i \neq j$. By Lemma $\ref{orth}$, we have $e_1 + \cdots + e_n = \{1\}$. Then
    
    \begin{align*}
     a \in ae_1 + \cdots + ae_k \subseteq b(e_1 + \cdots + e_k) + c (e_1 + \cdots + e_k) = b + c
    \end{align*}

  because $A$ is an hyperring.\\
  The general case follows easily by observing that given $t \in A$, $\mbox{spec}(A_t) \cong D(t) \subseteq \mbox{spec}(A)$ is a compact set.
  
  Now, assuming that $A$ is mono-hyperring, we prove that it is in fact geometric. Let $\{x_i\}_{i \in I}$, $\{e_i\}_{i \in I}$ family of elements of $A$ and $e \in A$ with $D(e) = \bigcup_{i \in I} D(e_i)$ and $x_i = x_j$ in $A_{e_ie_j}$ for all $i,j$. Since $D(e)$ is compact, we can assume $I = \{1, \ldots, n\}$ finite and $e, e_i \in \mbox{Id}(A)$. Furthermore, substituting $e_i$ by $e_i (e_1^c \cdots e_{i-1}^c)$, we can also assume that $e_ie_j = 0$ if $i \neq j$. Let $x \in x_1e_1 + \cdots + x_ne_n$. Then $xe_i = x_ie_i$ and so $x = x_i$ in $A_{e_i}$. 
\end{proof}

Let $A$ be an hyperring and $S \subseteq A$ a multiplicative set ($1 \in S$ and $S \cdot S \subseteq S$). Given $a,b \in A$, 
define $D_S (a,b) = \{x \in A \colon \mbox{ exists } u,v,s \in S \mbox{ such that } xu \in av+bs\}$.

\begin{definition}
 Let $A$ be a von Neumann hyperring and $S \subseteq A$. The set $S$ is called a von Neumann subgroup if
 
 \begin{enumerate}[$\textbf{.}$]
  \item $S$ is multiplicative.
  \item Given $a \in A$ idempotent, if $x \in D_S(a, a^c)$, then exist $s \in S$ such that $xs \in S$.
 \end{enumerate}

\end{definition}

\begin{theorem} \label{quoVN}
 Let $A$ be a von Neumann hyperring and $S \subseteq A$ a multiplicative set. Then $S$ is a von Neumann subgroup if, and only if, $A/_mS$ is a von Neumann geometric hyperring. In particular, $S$ is von Neumann subgroup if, and only if, $\overline{S}$ also is.
\end{theorem}

\begin{proof}
Let $\pi: A \to A/_mS$ be the natural projection. 

 $\Rightarrow:$ Since $A/_m S$ is von Neumann hyperring, by $\ref{geoVon}$ it is enough to prove that $x + x^c = \{1\}$ for all $x \in A/_mS$ idempotent. Let $x \in A$ such that $\pi(x)$ is an idempotent. Then $\pi(i(x)) = i(\pi(x)) = \pi(x)$. Thus given $\pi(z) \in \pi(x) + \pi(x)^c = \pi(i(x)) + \pi(i(x)^c)$, exists $u,v,w \in S$ such that $zu \in i(x)v + i(x)^cw$ and so $z \in D_S (i(x), i(x)^c)$. Since $S$ is a von Neumann subgroup, $\pi(z) = 1$. Therefore $A /_m S$ is a von Neumann geometric hyperring.\\
 $\Leftarrow$: Let $a \in A$ idempotent and take $x \in D_S(a,a^c)$. Then exists $s_1, s_2, s_3 \in S$ such that $xs_1 \in as_2 + a^cs_3$. Then $\pi (x) \in \pi(a) + \pi(a)^c$ but since
 $A/_mS$ is geometric, $\pi(x) = 1$ and so exist $s \in S$ such that $xs \in S$.
 
 The particular conclusion about $\overline{S}$ follows by Proposition $\ref{marshallQ}$.
\end{proof}

\begin{example} \label{nonGeometric}
The Marshall quotient usually not preserve geometric hyperrings. Let $A = \mathbb{R} \times \mathbb{R}$ and $S = \{1,(2,3),(2^2,3^2), \ldots\}$. Let $x = (2^2, 3)$ and $a = (1,0)$.
 Note that $a^2 = a$ is idempotent and $a^c = (0,1)$. Note yet that $x = a (2^2,3^2) + a^c (2,3)$ but there is no $s \in S$ such that $xs \in S$ and so $A/_mS$ is von Neumann hyperring
 which is not a mono-hyperring.
\end{example}

\begin{remark}
 If $A$ be a von Neumann GH and $S \subseteq \mbox{Id}(A)$ is a multiplicative set (or more generaly
 if $S$ is a multiplicative set such that for all $a,b \in S$ exist $c \in S$ such that $ac = bc$), then $A /_m S$ is also a geometric hyperring. If $A$ is a von Neumann
 hyperring and $S$ is a general multiplicative set it is possible to find the smaller $S'$ multiplicative such that $S \subseteq S'$ and $A /_m S'$ is geometric. This
 is the content of the next result.
\end{remark}

Let $A$ be a von Neumann hyperring. Consider the set

\begin{align*}
    S_u \defeq \bigcup \{e_1 + \cdots + e_n \colon \{e_1, \ldots, e_n\} \mbox{ is a partition of unity}\}.
\end{align*}

By Lemma $\ref{orth}$, it is a upward-union closed by multiplication.

\begin{corollary} \label{geoHullvN}
 Let $A$ be a von Neumann hyperring.
 Then $A/_m S_u$ is a von Neumann geometric hyperring and given a morphism $f \colon A \to B$ such that $B$ is a mono-multiring, exist unique map
 $\overline{f} \colon A/_m S_u \to B$ such that 
 
\[
  \begin{tikzcd}
    A \arrow{r} \arrow[swap]{dr}{f} & A /_m S_u \arrow{d}{\overline{f}} \\
     & B
  \end{tikzcd}
\]
 
 is a commutative diagram.
\end{corollary}

\begin{proof}
 In order to show that $A/_mS$ is geometric, we prove that $S$ is a von Neumann subgroup. Let $a \in A$ idempotent and take $x \in D_S(a,a^c)$. Then exists $s_1, s_2, s_3 \in S$ such that $xs_1 \in as_2 + a^cs_3$. Since $S_u$ is an upward-union, exist a partition of unity $\{e_1, \ldots, e_n\}$ such that
 $s_1,s_2,s_3 \in e_1 + \cdots + e_n$. Then $xs_1 \in ae_1 + \cdots + ae_n + a^ce_1 + \cdots a^c e_n \subseteq S$. Thus $S$ is a von Neumann subgroup.
 
 To conclude the desired universal property, we have to prove that for all map $f \colon A \to B$, $B$ mono-multiring, $f(S) = \{1\}$.
 Let $x \in S$ and $\{e_1, \ldots, e_n\}$ partition of unity with $x \in e_1 + \cdots + e_n$. Then $f(x) \in f(e_1) + \cdots + f(e_n)$. Since $\{f(e_1), \ldots, f(e_n)\}$ is a partition of unity in B and it is mono-multiring, we have $f(x) = 1$.
\end{proof}

Let $A$ be a geometric von Neumann hyperring and $a \in A$. Given $x,y \in a - a^c$, then $xi(a) = a = yi(a)$ and $xi(a)^c = -1 = yi(a)^c$. Thus $x \in xi(a) + xi(a)^c = yi(a) + yi(a)^c = \{y\}$. Therefore, we denote $\nabla(a)$ the unique element in $a - a^c$.

\begin{proposition} \label{nablai}
 Let $A$ be a geometric von Neumann hyperring and let $a \in A$.
 
 \begin{enumerate}[i)]
  \item $\nabla(a) \in A^{\times}$ and $a = i(a) \nabla(a)$.
  \item If $f \colon A \to B$ is a morphism of GvNH, given $a \in A$, then  $\nabla(f(a)) = f(\nabla(a))$.
  \item Given $x, y \in A$, $x = y$ if, and only if, $\nabla(i(x)) = \nabla(i(y))$ and $\nabla(x) = \nabla(y)$. In particular, if $e,f \in A$ are idempotents, then $\nabla(e)  = \nabla(f) \Leftrightarrow e = f$.
 \end{enumerate}

\end{proposition}

\begin{proof}

\begin{enumerate}[i)]
 \item It is enough to observe that given $b \in A$ with $a = a^2b$, then $ab$ is an idempotent and $D(a) = D(ab)$.
 
 \item We know that $i(a)^c$ is defined to be the unique element $x$ such that $x \in 1 - i(a)$ and $xi(a)=0$ (remark $\ref{vNremarks}$). Then taking some $b \in A$ with $a = a^2b$, we know that
 $i(a) = ab$ and so $xi(a) = 0$ if, and only if, $xa = 0$.
 
 \item Since $\nabla(a) \in a - a^c$, given $p \in D(a)$, we have $a^c \in p$ and so $\nabla(a) \notin p$; if $p \in D(a)^c = D(a^c)$ we also conclude that $\nabla(a) \notin p$. Thus by Proposition $\ref{sheaf2}$ $\nabla(a) \in A^{\times}$. To see that $a = i(a) \nabla(a)$, note that since $\nabla(a) \in a - a^c$, we have $\nabla(a)i(a) \in a i(a) - a^ci(a) = \{ai(a)\} = \{a\}$.
 \item Since $\nabla(a) \in a - a^c$, we have $f(\nabla(x)) \in f(a) - f(a^c) = f(a) - f(a)^c$ and so $\nabla(f(a)) = f(\nabla(a))$.

 \item $\Rightarrow:$ Immediate.\\
      $\Leftarrow:$ Assume that $\nabla(i(x)) = \nabla(i(y))$ and $\nabla(x) = \nabla(y)$. We only have to prove that if $e,f \in A$ are idempotents with $\nabla(e) = \nabla(f)$, then $e = f$ because this imples in our case that $i(x) = i(y)$ and so $x = i(x) \nabla(x) = i(y) \nabla(y) = y$.
      So assume $\nabla(e) = \nabla(f), e,f \in \mbox{Id}(A)$. Since $A$ is geometric, it is enough to show that for all $p \in \mbox{spec}(A)$ $e = f$ in $K_A(p)$. Fixed a prime ideal $p$, if $e = 1$ in $K_A(p)$, then $1 = \nabla(e) = \nabla(f)$ and thus $f = 1 = e$ in $K_A(p)$; if $e = 0$ in $K_A(p)$, then $-1 = \nabla(e) = \nabla(f)$ and thus $f = 0 = e$ in $K_A(p)$.
\end{enumerate}
\end{proof}

\begin{proposition} \label{GvNHR}
 Let $A$ be a multiring. The following are equivalent:
 
 \begin{enumerate}[i)]
    \item $A$ is real reduced multiring and for each $a \in A$, exists $x \in A$ such that $ax = 0$ and $x \in 1 - a^2$.
 
     \item $A$ is real reduced hyperring.
     
     \item $A$ is a geometric von Neumann hyperring and for all $a \in A$, $1 + a^2 = \{1\}$.
 \end{enumerate}
\end{proposition}

\begin{proof}
$i) \Rightarrow ii)$: Let $d \in ba + ca$. We have to prove that exists $f \in b + c$ such that $d = fa$. By hypothesis, exists $x \in 1 - a^2$ such that $ax = 0$. Let $l \in b + c$ and consider $f \in lx + da$. 

Using the Corollary $\ref{sepTheorem}$, the proof is completed if we prove that for all $\sigma \in \mbox{sper}(A)$, $\sigma (f) \in \sigma(b) + \sigma(c)$ and $\sigma(d) = \sigma(f)\sigma(a)$. Fix $\sigma \in \mbox{sper}(A)$. Note that if $\sigma(a) = 0$, then $\sigma(x) \in 1 - \sigma(a)^2 = \{1\}$ and if $\sigma(a) \neq 0$, $0 = \sigma(x)\sigma(a) = \sigma(x) \sigma(a)^2 = \sigma(x)$ and $\sigma(a)^2 = 1$.

\begin{itemize}
    \item $\sigma(f) \in \sigma(b) + \sigma(c)$.\\
    If $\sigma(a) = 0$, then $\sigma(x) = 1$ and $\sigma(f) \in \sigma(lx) + \sigma(da) = \{\sigma(l)\} \subseteq \sigma(b) + \sigma(c)$. If $\sigma(a) \neq 0$, then $\sigma(x) = 0$ and thus $\sigma(f) \in \sigma(lx) + \sigma(da) = \{\sigma(da)\} \in \sigma(b)\sigma(a)^2 + \sigma(c)\sigma(a)^2 = \sigma(b) + \sigma(c)$.
    
    \item $\sigma(d) = \sigma(fa)$. \\
    If $\sigma(a) = 0$, then $\sigma(d) \in \sigma(ba) + \sigma(ca) = \{0\}$ and thus $\sigma(d) = 0 = \sigma(fa)$. If $\sigma(a) \neq 0$, then $\sigma(x) = 1$ and $\sigma (fa) \in \sigma(lxa) + \sigma(da^2) = \{\sigma(d)\}$ because $ax = 0$.
\end{itemize}

$ii) \Rightarrow iii)$: Since $A$ is a RRM, then in the language of real semigroup $A$ is a Real Semigroup (see \cite{RRM}) and, by Theorem III.1.1 in \cite{DP2} $A$ is a geometric. On the other hand, by Proposition $\ref{cVn}$, $A$ is an von Neumann hyperring.

$iii) \Rightarrow i)$:
First we prove that $A$ is a real reduced multiring.

\begin{itemize}
 \item $\forall a \in A$, $a^3 = a$. \\
 By hypothesis, we have $1 + \nabla(a)^{-2} = \{1\}$. Thus 
 
 \begin{eqnarray*}
 \{\nabla(a)^2\} = \nabla(a)^2 (1 + \nabla(a)^{-2}) = 1 + \nabla(a)^2 = \{1\}.
 \end{eqnarray*}
 
 Thus $a^3 = (i(a) \nabla(a))^3 = i(a)\nabla(a) = a$. 
 
 \item $a + ab^2 = \{a\}$.\\
 Since $A$ is an hyperring, by hypothesis $a + ab^2 = a(1 + b^2) = \{a\}$.
 
 \item If $x,y \in a^2 + b^2$, then $x = y$. \\
 First note that, since $a^2,b^2 \in \mbox{Id}(A)$, we have $i(a) = a^2$. Note also that
 
 \begin{align*}
 xa^2,ya^2 \in & \  a^2 + b^2a^2 = \{a^2\} \\
 xa^c,ya^c \in & \ 0 + b^2a^c = \{b^2a^c\}.
 \end{align*}
 
 Thus $xa^2 = ya^2$, $xa^c  = ya^c$. Since $a^2 + a^c = \{1\}$, 
 
 \begin{eqnarray*}
 x \in x(a^2 + a^c) = xa^2 + xa^c = ya^2 + ya^c = y(a^2 + a^c) = \{y\}.
 \end{eqnarray*}
 
 \end{itemize}
 To complete the proof, note that given $a \in A$, the element $a^c$ satisfies $aa^c = 0$ and $a^c \in 1 - i(a) = 1 - a^2$.
\end{proof}

\begin{lemma}
Let $A$ be von Neumann hyperring and $T \subseteq A$ a proper preorder. Then $1 + T \subseteq A$ is a von Neumann subgroup.
\end{lemma}

\begin{proof}
Let $a \in A$ be an idempotent and $u \in a + a^c$. First note that by Proposition $\ref{sheaf2},iii)$ $u$ is invertible. Observing that $u \in \sum A^2 \subseteq T$ by hypothesis and $u^{-1} = u^{-2} u \in T$, the expression $(1 + u^{-1}) u = 1 + u$ assures the existence of $t \in 1 + T$ such that $ut \in 1 + T$.

Let $x \in D_{1 + T} (a, a^c)$. Then exists $s, u, v \in 1 + T$ such that $xs \in au + va^c$. Taking $u',v' \in T$ such that $u \in 1 + u', v \in 1 + v'$, we have
 
 \begin{align*}
 xs \in au + a^cv \subseteq (a + au') + (a^c + a^cv') &= (a + a^c) + (au' + a^cv'). \\
 \end{align*}
 
 Then exists $t' \in T$ such that $xst' \in t' + au't' + a^cv't' \in 1 + T$.
 Thus $1 + T \subseteq A$ is von Neumann subgroup. 
\end{proof}

\begin{theorem} \label{reprevN}
 Let $A$ be von Neumann hyperring and $T \subseteq A$ a proper preorder. Then the natural projection $\pi \colon A \to Q_T(A)$ induces an isomorphism $Q_T(A) \cong A /_m (1 + T)$.
\end{theorem}

\begin{proof}
 Let $\rho \colon A \to A/_m (1 + T)$ be the natural projection. By the preceding Lemma, $1 + T \subseteq A$ is von Neumann subgroup and thus by Theorem $\ref{quoVN}$ $A/_m (1 + T)$ is GvNH. On the other hand, given $x,a \in A$ such that $\rho(x) \in 1 + \rho(a)^2$ in $A/_m (1 + T)$, exists $t_1,t_2,t_3 \in 1 + T$ with $xt_1 \in t_2 + a^2t_3 \subseteq 1 + T$ and thus $\rho(x) = 1$. Therefore, by Proposition $\ref{GvNHR}$, the hyperring $A /_m (1 + T)$ is RRM. To see that $Q_T(A) \cong A /_m (1 + T)$ we will prove that $\rho \colon A \to A/_m (1 + T)$ satisfies the same universal property of $\pi \colon A \to Q_T(A)$ (Theorem $\ref{uniProMRR}$). Let $f \colon A \to R$ a multiring morphism with $R$ RRM such that $f(T) \subseteq \sum R^2 = R^2 = Id(R)$. Then $f(1 + T) \subseteq 1 + Id(R) = \{1\}$ and thus by Proposition $\ref{marshallQ}$ exists unique $\overline{f} \colon A/_m (1 + T) \to R$ such that $f = \overline{f} \circ \rho$, as desired.
\end{proof}

As the last result of this section, we analyse some constructions that $Q \colon \textbf{pvNH} \to \textbf{RRM}$ preserve in order to deduce a logical-preservation result (Theorem $\ref{theEleE}$).

Let $\langle M_i \colon \{\mu_{i,j} \colon i \leq j \mbox{ in } I\}\rangle$ be an inductive system of multrings and $\mu_i \colon M_i \to \varinjlim \mathcal{M}$ the inclusions morphisms. For each $i \in I$, let $S_i \subseteq M_i$ be a subset. We denote by $\varinjlim S_i \defeq \bigcup_{i \in I} \mu_i(S_i)$. If each $S_i \subseteq A_i$ is multiplicative and $\mu_{i,j} (S_i) \subseteq S_j$ when $i \leq j$ in $I$, then $\langle M_i/_m S_i \colon \{\overline{\mu_{i,j}} \colon i \leq j \mbox{ in } I\}\rangle$ is also an inductive system, where $\overline{\mu_{i,j}} \colon M_i/_m S_i \to M_j /_m S_j$ is the induced map.

\begin{lemma} \label{marshalQl}
\begin{enumerate}[i)]
    \item 
Let $\langle M_i \colon \{\mu_{i,j} \colon i \leq j \mbox{ in } I\}\rangle$ be an inductive system of multrings and $S_i \subseteq M_i$ multiplicative set for each $i \in I$ with $\mu_{i,j} (S_i) \subseteq S_j$ when $i \leq j$ in $I$. Then $\varinjlim (M_i/S_i) \cong (\varinjlim M_i) /_m \varinjlim S_i$.

\item Let $\{A_i\}_{i \in I}$ be a family of multirings and $S_i \subseteq A_i$ a multiplicative set for each $i \in I$. Then $\prod_{i \in I} (A_i /_m S_i) \cong (\prod_{i \in I} A_i) /_m (\prod_{i \in I} S_i)$.

\end{enumerate}
\end{lemma}

\begin{proof}
The proof can be found in \cite{RRM}.
\end{proof}

Let $\mathcal{L}_{pMulti} = \mathcal{L}_{multi} \cup \{T\} = \{\pi, \cdot, -, 0, 1, T\}$ the language of p-multirings where $T$ is a unary relation. Note that the pre-order axioms are geometric and thus exist inductive limits in $\textbf{pMulti}$. Consider the category $\textbf{pvNH}$ as faithfully subcategory of $\textbf{pMulti}$ of pairs $(A,T)$ where $A$ is a von Neumann hyperring.

 Let $\{(A_i, T_i)\}_{i \in I}$ be a family of p-multirings. Then $(\prod_{i \in I} A_i,\prod_{i \in I}T_i)$ is the product of the family $\{(A_i, T_i)\}_{i \in I}$ in the category $\textbf{pMulti}$. Since arbitrary product of von Neumann hyperring is again vNH and $\textbf{pvNH}$ is faithfully subcategory of $\textbf{pMulti}$, if the original family $\{(A_i, T_i)_{i \in I}\}$ is in $\textbf{pvNH}$, then $(\prod_{i \in I} A_i, \prod_{i \in I} T_i)$ is the product in $\textbf{pvNH}$.

Another useful construction in $\textbf{pMulti}$ is inductive limit. Let $\langle (A_i,T_i) \colon \{\mu_{i,j} \colon i \leq j \mbox{ in } I \}\rangle$ be an inductive system of p-multirings. Then the $\varinjlim_{i \in I} (A_i,T_i)$ exists in $\textbf{pMulti}$ because the first-order axioms of p-multrings are geometric. Furtheremore, we have $\varinjlim (A_i, T_i) \cong (\varinjlim A_i, \varinjlim T_i)$. As in product case, the same considerations abou inductive limits in $\textbf{pMulti}$ holds in $\textbf{pvNH}$.

From now on we prove that the functor $Q \colon \textbf{pvNH} \to \textbf{RRM}$ preserves elementar equivalence (two structures are elementar equivalent if they satisfies the same first-order sentences). We need the classical Keisler-Shelah Theorem that gives an algebraic characterization of elementar equivalence in terms of ultrapowers.

\begin{fact}[Keisler-Shelah Theorem's]
Let $\mathcal{L}$ be a first-order language and $A,B$ two models of $\mathcal{L}$. Then $A \equiv B$ if, and only if, exists set $I$ and ultrafilter $U$ over $I$ such that $\prod_{i \in I} A / U \cong \prod_{i \in I} B / U$.
\end{fact}

\begin{proof}
The proof can be found in \cite{CK}.
\end{proof}

\begin{remark} \label{ultrap}
 Let $\mathcal{L}$ be a first-order language and $\{A_i\}_{i \in I}$ a family of $\mathcal{L}$-models. If $U$ is an ultrafilter over $I$, then
the ultraproduct $\prod_{i \in I} A_i / U$ is an inductive limit of products. More precisely, define $\mathcal{F} \colon U \to \mathcal{L}-\mbox{models}$ by $\mathcal{F}(A) = \prod_{i \in A} A_i$ and if $A \subseteq B$ in $U$, $\mathcal{F}_{B,A} \colon \mathcal{F}(B) \to \mathcal{F}(A)$ is the natural projection. Then $\mathcal{F}$ is a contravariant functor over a left-directed set and $\varinjlim \mathcal{F} \cong \prod_{i \in I} A_i / U$ (see more details in Propostion $\ref{filtColim}$ or \cite{CK}). 
\end{remark}

\begin{theorem} \label{theEleE}
  Let $(A,T),(B,P)$ be von Neumann p-hyperrings. If $(A,T) \equiv_{\mathcal{L}_{pMulti}} (B,P)$, then $Q_T(A) \equiv_{\mathcal{L}_{multi}} Q_P(B)$.
\end{theorem}

\begin{proof}

In order to prove that $Q \colon \textbf{pvNH} \to \textbf{RRM}$ preserves elementar equivalence, by Keisler-Shelah Theorem and Remark $\ref{ultrap}$ we have to shows that $Q$ preserves inductive limit and arbitrary products.

\textbf{Inductive Limits} \footnote{Another way to prove this is to observe that the functor $Q \colon \textbf{pMulti} \to \textbf{RRM}$ is left-adjoint to the inclusion functor $\textbf{RRM} \to \textbf{pMulti}$ by Theorem $\ref{uniProMRR}$ and thus preserve inductive limits.}

Let $\langle (A_i, T_i) \colon \{\mu_{i,j} \colon i \leq j \mbox{ in } I \}\rangle$ be an inductive system of pvNH. Note that $1 + \varinjlim T_i = \varinjlim (1 + T_i)$. Using the Theorem $\ref{reprevN}$ and Lemma $\ref{marshalQl}$, we have

\begin{align*}
    Q(\varinjlim (A_i,T_i)) &\cong Q(\varinjlim A_i, \varinjlim T_i) \cong \varinjlim A_i /_m (1 + \varinjlim T_i) \\
    & = \varinjlim A_i /_m \varinjlim (1 + T_i) \cong \varinjlim A_i /_m (1 + T_i) \\
    & \cong \varinjlim Q(A_i, T_i).
\end{align*}

\textbf{Products:}\\
     Let $\{(A_i,T_i)\}_{i \in I}$ be a family of pvNH. 
     Then by Theorem $\ref{reprevN}$ and Lemma $\ref{marshalQl}$
     
     \begin{align*}
         Q(\prod_{i \in I} (A_i, T_i)) &\cong Q(\prod_{i \in I}A_i, \prod_{i \in I} T_i)  \cong (\prod_{i \in I} A_i) /_m (1 + \prod_{i \in I}T_i) \\
         & = (\prod_{i \in I} A_i) /_m \prod_{i \in I}(1 + T_i) \cong \prod_{i \in I} A_i /_m (1 + T_i) \\
         & \cong \prod_{i \in I} Q(A_i, T_i).
     \end{align*}
\end{proof}
     
     Let $A$ be a von Neumann hyperring. For each $x \in \sum A^2$, consider the number $P(x) = \mbox{min}\{n \in \mathbb{N} \colon \mbox{exists } x_1, \ldots, x_n \in A \mbox{ such that } x \in x_1^2 + \ldots + x_n^2 \}$. Also consider
     
     \begin{align*}
         P(A) = sup \{n(x) \colon x \in \sum A^2\}
     \end{align*}

    called the Pythagorean number of $A$.
    It is possible to prove that $P(A) = sup \{P(K_A(q)) \colon q \in \mbox{spec}(A)\}$ and $P(A) < \infty$ if, and only if, $P(K_A(q)) < \infty$ for every $q \in \mbox{spec}(A)$.

     \begin{corollary}
     Let $A,B$ semi-real von Neumann hyperrings and assume $P(A), P(B) < \infty$. If $A \equiv_{\mathcal{L}_{multi}} B$, then $Q(A) \equiv_{\mathcal{L}_{multi}} Q(B)$.
     \end{corollary}
     
     \begin{proof}
     Since $P(A) < \infty$, the preorder the set $\sum A^2$ is definible by finitary first-order formuluas in the language $\mathcal{L}_{multi}$ (the same true for $B$). Thus if $A \equiv_{\mathcal{L}_{multi}} B$, then $(A, \sum A^2) \equiv_{\mathcal{L}_{pMulti}} (B, \sum B^2)$ and by Theorem $\ref{theEleE}$ we have $Q(A) \equiv_{\mathcal{L}_{multi}} Q(B)$.
     \end{proof}

\end{proof}

\section{The von Neumann Hull for Multiring}

In this section, it is built a geometric von Neumann hyperring hull for every multiring. The intention is to generalize the von Neumann hull for rings (\cite{AM}).

\begin{definition}
Let $A$ be a multiring. A function $f \colon \mbox{spec}(A) \to \bigsqcup_{p \in spec(A)} K_A(p)$ is called a constructible section if for all $p \in \mbox{spec}(A)$, $f(p) \in K_A(p)$ and given $p \in \mbox{spec}(A)$, exists clopen $U \subseteq (\mbox{spec}(A))_{const}$ and $x,y \in A$ such that

\begin{itemize}
    \item $p \in U \subseteq D(y)$.
    \item For all $q \in U$, $f(q) = \frac{x}{y} \in K_A(q)$.
\end{itemize}

The set of all constructible sections are denoted by $V(A)$ and $v_A \colon A \to V(A)$ is the natural map given by $v_A(a)(p) = \frac{a}{1} \in K_A(p)$ for every $p \in \mbox{spec}(A)$ and $a \in A$.
\end{definition}

\begin{remark}
Let $A$ be a multiring and $f \in V(A)$ a constructible section. Note that since $(\mbox{spec}(A))_{cons}$ is compact, exists $U_1, \ldots, U_n \subseteq (\mbox{spec}(A))_{const}$ clopens and $x_i,y_i \in A$, $i=1, \ldots, n$, such that

\begin{enumerate}[.]
    \item $\mbox{spec}(A) = \bigcup_{i=1}^n U_i$, $U_i \subseteq D(y_i)$.
    
    \item For all $q \in U_i$, $f(q) = \frac{x_i}{y_i} \in K_A(q)$.
\end{enumerate}

More than that, if we substitute $U_i$ by $U'_i = U_i \cap U_1^c \cap \cdots \cap U_{i-1}^c$, we can assume that the family $\{U_1, \ldots, U_n\}$ is disjoint.
\end{remark}

\begin{proposition} \label{intHull}
 Let $A$ is a geometric von Neumann hyperring, then $v_A \colon A \to V(A)$ is a bijection.
\end{proposition}

\begin{proof}
    Since $A$ is geometric, it follows by Theorem $\ref{geoVon}$ that $v_A$ is injective. On the other hand, given $f \in V(A)$, exists $U_1, \ldots, U_n \subseteq \mbox{spec}(A)$ clopens and $x_i,y_i \in A$ such that
    
    \begin{enumerate}[.]
        \item $\mbox{spec}(A) = \bigsqcup_{i=1}^n U_i$ and $U_i \subseteq D(y_i)$.
        
        \item For all $p \in U_i$, $f(p) = \frac{x_i}{y_i} \in K_A(p)$.
    \end{enumerate}
    
    Because $A$ is von Neumann, exists $e_i \in A$ idempotent such that $U_i = D(e_i)$. Consider the element $a \in e_1x_1\nabla(y_1)^{-1} + \cdots + e_nx_n \nabla(y_n)^{-1}$. Then for all $p \in U_i$, in $K_A(p)$ we have $v_A(a)(p) = x_i \nabla(y_i)^{-1} = \frac{x_i}{y_i} = f(p)$. Thus $v_A(a) = f$.
    
    \end{proof} 

The preceding Proposition suggest that $V(A)$ can be a candidate do von Neumann hull. Consider in $V(A)$ the product punctually defined and the multivaluated sum defined in a similar manner: given $f,g,h \in V(A)$

\begin{align*}
    f \in g + h \Leftrightarrow \mbox{ for every } p \in \mbox{spec}(A) \ f(p) \in g(p) + h(p) \mbox{ in } K_A(p) \\
\end{align*}

First we need to prove that $V(A)$ is a multiring and in the sequence that it is von Neumann hyperring.

\begin{lemma} \label{lemma1}
Let $A$ be a multiring and $x_i,y_i \in A$, $i=1,2,3$. Given $p \in D(y_1y_2y_3)$ such that $\frac{x_1}{y_1} \in \frac{x_2}{y_2} + \frac{x_3}{y_3}$ in $K_A(p)$, exist a clopen $U \subseteq D(y_1y_2y_3)$ with $p \in U$ such that for all $q \in U$, $\frac{x_1}{y_1} \in \frac{x_2}{y_2} + \frac{x_3}{y_3}$ in $K_A(q)$.
\end{lemma}

\begin{proof}
If $\frac{x_1}{y_1} \in \frac{x_2}{y_2} + \frac{x_3}{y_3}$ in $K_A(p) = \mbox{ff}(A/p)$, by definition exists $s \notin p$ and $t \in p$ such that

\begin{align*}
    x_1 (y_2y_3)s \in x_2 (y_1y_3)s + x_3 (y_1y_2)s + t \mbox{ in } A.
\end{align*}

Let $U = D(y_1y_2y_3) \cap D(s) \cap D(t)^c$. Then $U \subseteq (\mbox{spec}(A))_{const}$ is clopen with $p \in U$ and by definition for all $q \in U$ we have $\frac{x_1}{y_1} \in \frac{x_2}{y_2} + \frac{x_3}{y_3}$ in $K_A(q)$.

\end{proof}

\begin{theorem}
Let $A$ be a multiring. Then $V(A)$ is a multiring and $v_A \colon A \to V(A)$ a multiring morphism.
\end{theorem}

\begin{proof}
It is not difficult to see that $V(A)$ satisfies all multiring axiom except the associative property. So let $f,g,h,z,t \in V(A)$ with 

\begin{align*}
    z \in f + t, t \in g + h.
\end{align*}

We should find $t' \in f + g$ such that $z \in t' + h$. For each $p \in \mbox{spec}(A)$, since $z(p) \in f(p) + t(p), t(p) \in g(p) + h(p)$ in $K_A(p)$, exist $x,y \in A$ such that
$p \in D(y)$ and

\begin{align*}
z(p) \in \frac{x}{y} + h(p), \frac{x}{y} \in f(p) + g(p) \mbox{ in } K_A(p).
\end{align*}

By the Lemma $\ref{lemma1}$, exist clopen $U \subseteq (\mbox{spec}(A))_{const}$ with $p \in U$ such that for all $q \in U$, $z(q) \in \frac{x}{y} + h(q), \frac{x}{y} \in f(q) + g(q) \mbox{ in } K_A(q)$. Since $(\mbox{spec}(A))_{const}$ is Boolean, exists $U_1, \ldots, U_n$ clopens and $x_i,y_i \in A$, $i=1, \ldots, n$ such that

\begin{enumerate}[.]
    \item $\mbox{spec}(A) = \bigcup_{i=1}^n U_i$, with $U_i \subseteq D(y_i)$.
    
    \item For each $p \in U_i$, $z(p) \in \frac{x_i}{y_i} + h(p), \frac{x_i}{y_i} \in f(p) + g(p) \mbox{ in } K_A(p)$.
\end{enumerate}

Substituing $U_i$ by $U_i \cap U_1^c \cap \cdots \cap U_{i-1}^c$, we can assume that $U_i \cap U_j = \emptyset$ for $i \neq j$. Define the function $t' \colon \mbox{spec}(A) \to \bigsqcup_{p \in spec(A)}K_A(p)$ given by 

\begin{align*}
\mbox{If }p \in U_i, \mbox{ then } t'(p) = \frac{x_i}{y_i}.    
\end{align*}

Then $t' \in V(A)$ and $z \in t' + h$, with $t' \in f + g$.

\end{proof}

Given $f \in V(A)$, consider the functions
$i(f), \nabla(f), f^c \colon \mbox{spec}(A) \to \bigsqcup_{p \in spec(A)} K_A(p)$ given by

\begin{itemize}
    \item 
    
   \begin{equation*}
    i(f)(p) =
    \begin{cases*}
      1 & if $f(p) \neq 0$ \\
      0 & if $f(p) = 0$.
    \end{cases*}
  \end{equation*}
  
  \item
  \begin{equation*}
    \nabla(f)(p) =
    \begin{cases*}
      f(p) & if $f(p) \neq 0$ \\
      -1 & if $f(p) = 0$.
    \end{cases*}
  \end{equation*}
  
  \item
  
  \begin{equation*}
    f^c (p) =
    \begin{cases*}
      0 & if $f(p) \neq 0$ \\
      1 & if $f(p) = 0$.
    \end{cases*}
  \end{equation*}
  
\end{itemize}

\begin{lemma} \label{lemma3}
Let $A$ be a multiring and $f,g \in V(A)$. Then

\begin{enumerate}[i)]
    \item $i(f),\nabla(f),f^c \in V(A)$ and $i(f)^2 = i(f), (f^c)^2 = f^c$, $\nabla(f) \in V(A)^{\times}$.
    \item $f^c = i(f)^c$.
    \item $ff^c = 0$ and $f^c \in 1 - i(f)$.
    \item If $f^2 = f$, then $fg + f^cg = \{g\}$.
    \item $f = \nabla(f) i(f)$.
\end{enumerate}
\end{lemma}

\begin{proof}
Straightforward.
\end{proof}

\begin{theorem}
Let $A$ be a multiring. Then $V(A)$ is a geometric von Neumann hyperring.
\end{theorem}

\begin{proof}
First we prove that $V(A)$ is an hyperring. Let $x \in fg + fh$. We need to find $y \in g + h$ such that $x = fy$. Let $l \in g + h$ and take

\begin{align*}
    y \in f^cl + \nabla(f)^{-1}x.
\end{align*}

Note that since $f^cx \in f^cfg + f^cfh = \{0\}$, we have by Lemma $\ref{lemma3}$ that $0 = f^cx \in x - xi(f)$ and so $x = xi(f)$. Then $yf = i(f) \nabla(f)^{-1}xf = i(f)x = x$. On the other hand, we have $f^cl \in f^cg + f^ch$ and $\nabla(f)^{-1} \in \nabla(f)^{-1}fg + \nabla(f)^{-1}fh = i(f)g + i(f)h$. Then

\begin{align*}
    y \in (f^cg + f^ch) + (i(f)g + i(f)h) = (f^cg + i(f)g) + (f^ch + i(f)h) = g + h.
\end{align*}

Therefore, $V(A)$ is an hyperring. Lastly, $V(A)$ is a geometric von Neumann hyperring because given $f,g \in V(A), g^2=g$, we have by Lemma $\ref{lemma3}$

\begin{enumerate}[.]
\item $(i(f) \nabla(f)^{-1}) f^2 = (i(f) \nabla(f)^{-1}) (i(f) \nabla(f))^2 = i(f) \nabla(f) = f$.

\item $g + g^c = \{1\}$
\end{enumerate}
\end{proof}

\begin{remark} \label{remarktopAlg}
Since $\{D(a) \cap D(b_1)^c \cap \cdots \cap D(b_n)^c \colon n \geq 1, a,b_1, \ldots, b_n \in A\}$ is base of $(\mbox{spec}(A))_{const}$, given $f \in V(A)$ and $p \in \mbox{spec}(A)$, exists $a,b_1, \ldots, b_n, x,y \in A$ such that

\begin{enumerate}[i)]
    \item $p \in D(a) \cap D(b_1)^c \cap \cdots \cap D(b_n)^c \subseteq D(y).$
    
    \item For all $q \in D(a) \cap D(b_1)^c \cap \cdots \cap D(b_n)^c, f(q) = \frac{x}{y} \mbox{ in } K_A(q)$.
\end{enumerate}

Furthermore assuming that $D(a) \cap D(b_1)^c \cap \cdots \cap D(b_n)^c \subseteq D(y)$, the item $ii)$ is equivalent to 

\begin{align*}
v_A(y)f i(v_A(a)) v_A(b_1)^c \cdots v_A(b_n)^c = v_A(x)i(v_A(a)) v_A(b_1)^c \cdots v_A(b_n)^c.
\end{align*}
\end{remark}

\begin{lemma} \label{univ(f)}
Let $f \colon A \to B$ a multiring morphism. Let $g,h \colon V(A) \to V(B)$ two morphism such that

\[\begin{tikzcd}
        A \arrow[r, "f"] \arrow[d, "v_A"] & B \arrow[d, "v_B"] \\
        V(A)\arrow[r, bend left=18, "g"{below}]
  \arrow[r, bend right=18, "h"{below}] & V(B)\\
    \end{tikzcd}\]

is a commutative diagram. Then $g=h$.
\end{lemma}

\begin{proof}
Let $t \in V(A)$ and $p \in \mbox{spec}(B)$. We need to prove that $g(t)(p) = h(t)(p)$. Let $q = f^{-1}(p)$. Since $t$ is a section, exists $a,b_1, \ldots, b_n ,x,y \in A$ such that

\begin{enumerate}[.]
    \item $q \in D(a) \cap D(b_1) \cap \cdots \cap D(b_n)^c \subseteq D(y)$.
    
    \item For all $o \in D(a) \cap D(b_1)^c \cap \cdots D(b_n)^c$, $t(o) = \frac{x}{y} \in K_A(o)$.
\end{enumerate}

The last item implies that $v_A(y)t i(v_A(a)) v_A(b_1)^c \cdots v_A(b_n)^c = v_A(x)i(v_A(a)) v_A(b_1)^c \cdots v_A(b_n)^c$. Note that, since $g$ is a morphism, $g(i(v_A(a))) = i(g(v_A(a))) = i(v_B(f(a)))$ and $g(v_A(b_i)^c) = g(v_A(b_i))^c = v_B(f(b_i))^c$ for $i=1, \ldots, n$. Thus

\begin{align*}
    v_B(f(y))g(t) i(v_B(f(a))) v_B(f(b_1))^c \cdots v_B(f(b_n))^c &= g(v_A(y)) g(t) g(i(v_A(a))) g(v_A(b_1)^c) \cdots g(v_A(b_n)^c) \\
        & = g(v_A(y)t i(v_A(a)) v_A(b_1)^c \cdots v_A(b_n)^c) \\
        & = g(v_A(x)i(v_A(a)) v_A(b_1)^c \cdots v_A(b_n)^c) \\
        & = v_B(f(x))i(v_B(f(a))) v_B(f(b_1))^c \cdots v_B(f(b_n))^c.
\end{align*}

By the remark $\ref{remarktopAlg}$, the above equality is equivalent to

\begin{enumerate}[.]
    \item For all $o \in D(f(a)) \cap D(f(b_1))^c \cap \cdots \cap D(f(b_n))^c$, $g(t)(o) = \frac{f(x)}{f(y)} \in K_B(o)$.
\end{enumerate}

In particular, since $p \in D(f(a)) \cap D(f(b))^c$, we have $g(t)(p) = \frac{f(x)}{f(y)}$. Applying the same argument for $h$, we conclude that $h(t)(p) = \frac{f(x)}{f(y)} = g(t)(p)$. Since $p \in \mbox{spec}(A)$ was arbitrary, $g(t) = h(t)$. Thus $g = h$.

\end{proof}

\begin{theorem} \label{uniPGvNH}
Let $A, B$ be multirings.

\begin{enumerate}[i)]
    \item If $f \colon A \to B$ is a multiring morphism, then exist unique morphism $V(f) \colon V(A) \to V(B)$ such that
    
    \[\begin{tikzcd}
        A \arrow[r, "f"] \arrow[d] & B \arrow[d] \\
        V(A) \arrow[r, "V(f)"] & V(B)\\
    \end{tikzcd}\]

    is a commutative diagram. Furthermore, $V \colon \textbf{Multi} \to \textbf{GvNH}$ is a functor.
    
    \item Given a morphism $f \colon A \to B$, where $B$ is a GvNH, exist unique morphism $\overline{f} \colon V(A) \to B$ such that 
    
      \[\begin{tikzcd}
        A \arrow[r, "v_A"] \arrow[dr, "f"] & V(A) \arrow[d, "\overline{f}"] \\
         & B\\
\end{tikzcd}\]
    
    is a commutative diagram. In other words, the functor $V \colon \textbf{Multi} \to \textbf{GvNH}$ is a left-adjoint functor to the inclusion functor $\textbf{GvNH} \to \textbf{Multi}$.
\end{enumerate}
\end{theorem}

\begin{proof}
\begin{enumerate}[i)]
    \item Given morphism $f \colon A \to B$, consider the induced map $\mbox{spc}(f) \colon \mbox{spec}(B) \to \mbox{spec}(A)$ and given $p \in \mbox{spec}(B)$, the morphism $K_A(\mbox{spc}(f)(p)) \to K_B(p)$ (also denoted by $f$). Given $a \in V(A)$ function $a \colon \mbox{spec}(A) \to \bigsqcup_{q \in \mbox{spec}(A)} K_A(q)$, consider $\overline{f}(a) \colon \mbox{spec}(B) \to \bigsqcup_{p \in \mbox{spec}(B)} K_B(p)$ defined by $\overline{f}(a)(p) = f(a(\mbox{spc}(f)(p)))$. Note that $\overline{f}(a) \in V(B)$, that is, it is a section. In fact, given $p \in \mbox{spec}(B)$, exists clopen $U \subseteq \mbox{spec}(A)$ and $x,y \in A$ with
    
    \begin{enumerate}[.]
        \item $\mbox{spc}(f)(p) \in U$ and $U \subseteq D(y)$.
        
        \item For all $q \in U$, $a(q) = \frac{x}{y} \in K_A(q)$.
    \end{enumerate}
    
    Then $V = \mbox{spc}(f)^{-1}(U) \subseteq \mbox{spec}(B)$ is a clopen and
    
    \begin{enumerate}[.]
        \item $p \in V$ and $V \subseteq D(f(y))$.
        
        \item For all $q \in V$, $\overline{f}(a)(q) = \frac{f(x)}{f(y)} \in K_B(q)$.
    \end{enumerate}
    
    Thus we have a function $\overline{f} \colon V(A) \to V(B)$. It is easy to see that $\overline{f}(1) = 1, \overline{f}(-1) = -1$. Furthermore, given $a,b \in V(A)$, we have for $p \in \mbox{spec}(B)$
    
    \begin{align*}
        \overline{f}(ab)(p) &= f((ab) (\mbox{spc}(f)(p)))  = f(a(\mbox{spc}(f)(p)) b(\mbox{spc}(f)(p))) \\
        &= f(a(\mbox{spc}(f)(p))) f(b(\mbox{spc}(f)(p))) = \overline{f}(a)(p) \overline{f}(b)(p).
    \end{align*}
    
    Thus $\overline{f}(ab) = \overline{f}(a) \overline{f}(b)$. Also if $a \in b + c$ in $V(A)$, given $p \in \mbox{spec}(B)$, since $a(\mbox{spc}(f)(p)) \in b(\mbox{spc}(f)(p)) + c(\mbox{spc}(f)(p))$ in the hyperfield $K_A(\mbox{spc}(f)(p))$, we have  $\overline{f}(a)(p) \in \overline{f}(b)(p) + \overline{f}(c)(p)$ in the hyperfield $K_B(p)$. Thus $\overline{f}(a) \in \overline{f}(b) + \overline{f}(c)$ in $V(B)$. Therefore, $\overline{f} \colon V(A) \to V(B)$ is a multiring morphism. The uniqueness of $\overline{f}$ is stated and proved in the Lemma $\ref{univ(f)}$.
    \item First let prove the following claim:
    
    \begin{claim*}
    If $A$ is a geometric von Neumann hyperring, then $v_A \colon A \to V(A)$ is an isomorphism.
    \end{claim*}
     
    \begin{proof}
    By Proposition $\ref{intHull}$, $v_A \colon A \to V(A)$ is a bijection. On the other hand, by Theorem $\ref{geoVon}$, given $a,b,c \in A$, $a \in b + c$ if, and only if, $v_A(a) \in v_A(b) + v_A(c)$.  Thus $v_A$ is an isomorphism.
    \end{proof}
    Now if $f \colon A \to B$ is a morphism, $B$ GvNH, then consider $v(f) \colon V(A) \to V(B)$. By the claim above, $v_B$ is isomorphism and we can consider $\overline{f} = v_B^{-1} \circ v(f)$. Note that
    $\overline{f} \circ v_A = v_B^{-1} \circ (v(f) \circ v_A) = v_B^{-1} \circ (v_B \circ f) = f$. If $\overline{f}' \colon V(A) \to B$ is another morphism with $\overline{f}' \circ v_A = f$, then
    $v_B \circ \overline{f'} \circ v_A = v_B \circ f$. Thus by Lemma $\ref{univ(f)}$, $v_B \circ \overline{f}' = v(f)$ and so $\overline{f}' = v_B^{-1} \circ v(f) = \overline{f}$.
\end{enumerate}
\end{proof}

\begin{theorem} \label{algV(A)}
Let $A$ be a multiring and consider $v \colon A \to V(A)$.

\begin{enumerate}[i)]
    \item The induced map $\mbox{spec}(V(A)) \to (\mbox{spec}(A))_{const}$ is homeomorphism.
    
    \item The induced map $\mbox{sper}(V(A)) \cong (\mbox{sper}(A))_{const}$ is homeomorphism.
    
    \item Let $p \in \mbox{spec}(V(A))$ and $q = v^{-1} (p)$. The induced map $K_A(q) \to K_{V(A)}(p)$ is isomorphism.
\end{enumerate}

\end{theorem}

\begin{proof}

\begin{enumerate}[i)]
    \item Using the characterization of prime ideals as morphism to Krasner hyperfield $\mathbb{K} = \{0,1\}$, the universal property in Theorem $\ref{uniPGvNH}$ imples that the induced map $\mbox{spec}(V(A)) \to \mbox{spec}(A)$ is a bijection. But by since $V(A)$ is von Neumann hyperring, the topological space $\mbox{sper}(V(A))$ is Boolean and thus $\mbox{sper}(V(A)) \to (\mbox{sper}(A))_{const}$ is homeomorphism.
    
    \item Using the hyperfield $3 = \{-1,0,1\}$, the proof follows similarly to item $i)$.
    
    \item Consider the inclusions maps $i \colon A \to K_A(q)$ and $j \colon V(A) \to K_{V(A)} (p)$. By propositions $\ref{idealq}$ and $\ref{localization}$, if $f,g \colon K_A(q) \to K$, $K$ hyperfield, with $f \circ i = g \circ i$, then $f = g$. An analogous property holds for the map $j \colon V(A) \to K_{V(A)} (p)$. Consider the following commutative diagram:

    \[\begin{tikzcd}
        A \arrow[r, "v"] \arrow[d, "i"] & V(A) \arrow[d, "j"] \\
        K_A(q) \arrow[r, "v_q"] & K_{V(A)}(p)\\
    \end{tikzcd}\]

    Furtheremore, by Theorem $\ref{uniPGvNH}$, exist morphism $f \colon V(A) \to K_A(q)$ such that $f \circ v = i$. Note that $q = i^{-1} (0) = v^{-1} (f^{-1} (0))$ and thus $f^{-1}(0) = p$ because $v^{-1} \colon \mbox{spec}(V(A)) \to \mbox{spec}(A)$ is injective. Then by Proposition $\ref{idealq}$ exist $f_p \colon K_{V(A)} (p) \to K_A(q)$ such that

    \[\begin{tikzcd}
        A \arrow[r, "v"] \arrow[dr, "i"] & V(A) \arrow[d, "f"] \arrow[r, "j"] & K_{V(A)}(p) \arrow[dl, "f_p"]\\
                        & K_A(q) &                                  \\
    \end{tikzcd}\]
    
    is a commutative diagram.
    
    Then, combining this with the preceding diagram, we have $f_p \circ v_q \circ i = f_p \circ j \circ v = f \circ v = i$. Then $f_p \circ v_q = Id_{K_A(q)}$; on the other hand, we have $v_q \circ f_p \circ j \circ v = v_q \circ f \circ v = v_q \circ i  = j \circ v$. Thus, by Theorem $\ref{uniPGvNH}$, $v_q \circ f_p \circ j = j$ and then $v_q \circ f_p = Id_{K_{V(A)}(p)}$.

\end{enumerate}

\end{proof}

\begin{lemma} \label{RRMHFF}
Let $A$ be a real reduced multiring. Then for all $p \in \mbox{spec}(A)$, $K_A(p)$ is a real reduced hyperfield.
\end{lemma}

\begin{proof}
Since for all $x \in A$, $x^3 = x$, the same is true for $K_A(p)$; but because it is a hyperfield, if $a \in K_A(p)$ is non-zero, then $a^2 = 1$. On the other hande, given $x \in 1 + 1$ in $K_A(p)$, by definition exist $i \notin p$ and $x' \in A$ such that $x' - x \cap p \neq \emptyset$ and $x'i \in i + i$. Then $x'i = i$ and in $K_A(p)$ we have $x = x' = 1$. Thus $K_A(p)$ is an real reduced hyperfield.
\end{proof}

\begin{proposition} \label{GvnHRRM}
 Let $A$ be a multiring.
 
 \begin{enumerate}[i)]
     \item If $A$ is a real reduced multiring, then $V(A)$ also is.
     
     \item If $A$ is a von Neumann hyperring, then $Q(A)$ also is.
 \end{enumerate}
\end{proposition}

\begin{proof}
\begin{enumerate}[i)]
    \item For every $p \in \mbox{spec}(V(A))$, we have by Theorem $\ref{algV(A)}$ that $K_{V(A)}(p) \cong K_A(q)$, where $q = v_A^{-1} (p)$. Thus by Lemma $\ref{RRMHFF}$ $K_{V(A)}(p)$ is real reduced hyperfield. But since $1 + a^2 = \{1\}$ in $K_A(p)$ for every $p \in \mbox{spec}(V(A))$, the same is true for $V(A)$ because it is geometric. The Proposition $\ref{GvNHR}$ implies that $V(A)$ is real reduced hyperring.
    
    \item The result is a direct consequence of Theorem $\ref{reprevN}$ proof because if $A$ is a vNH, then  $Q(A) \cong A /_m 1 + \sum A^2$ is a GvNH.
\end{enumerate}

\end{proof}

\begin{theorem} \label{QVVQ}
Let $A$ be a semi-real multiring and consider the morphisms $\pi_A \colon A \to Q(A)$, $v_A \colon A \to V(A)$ and $v(\pi_A) \colon V(A) \to V(Q(A))$. Then $V(A)$ is semi-real and the induced map $Q(V(A)) \to V(Q(A))$ is isomorphism.
\end{theorem}

\begin{proof}
Since $A$ is semi-real, by Proposition $\ref{basicReal}$ and Theorem $\ref{algV(A)}$ we have $\emptyset \neq (\mbox{sper}(A))_{const} \cong \mbox{sper}(V(A))$ and thus $V(A)$ semi-real.
Now, to prove the induced map $Q(V(A)) \to V(Q(A))$ is an isomorphism, we proceed with several applications of universal properties described in Theorem $\ref{uniPGvNH}$ and Theorem $\ref{uniProMRR}$. First note that since $V(Q(A))$ is real reduced multiring by Proposition $\ref{GvnHRRM}$, the universal property of $V(A) \to Q(V(A))$ (Theorem $\ref{uniProMRR}$) guarantee the existence of $f \colon Q(V(A)) \to V(Q(A))$ such that 

    \[
  \begin{tikzcd}
    & V(A) \arrow{dr}{v(\pi_A)} \arrow{dl}{\pi_{V(A)}} & &&& Q(A) \arrow{dr}{Q(v_{A})} \arrow{dl}{v_{Q(A)}} & \\ 
  Q(V(A)) \arrow{rr}{f} && V(Q(A)) && V(Q(A)) \arrow{rr}{g} && Q(V(A))
  \end{tikzcd}
  \]

is a commutative diagram. On the other hand, considering the morphism $Q(v_A) \colon Q(A) \to Q(V(A))$, by Proposition $\ref{GvnHRRM}$ we have that $Q(V(A))$ is a GvNH and thus the universal property of $Q(A) \to V(Q(A))$ (Theorem $\ref{uniPGvNH}$) ensures the existence of $g \colon V(Q(A)) \to Q(V(A))$ such that $g \circ v_{Q(A)} = Q(v_A)$.

Now we have to prove that $g \circ f = Id_{Q(V(A))}$ and $f \circ g = Id_{V(Q(A))}$. Note that by Theorem $\ref{uniPGvNH}, i)$, we have $v(\pi_A) \circ v_A = v_{Q(A)} \circ \pi_A$ and by Theorem $\ref{uniProMRR}, i)$ $Q(v_A) \circ \pi_A = \pi_{V(A)} \circ v_A$.

\[
\begin{tikzcd}
A \arrow{r}{\pi_A} \arrow{d}{v_A} & Q(A) \arrow{d}{v_{Q(A)}} && A \arrow{r}{v_A} \arrow{d}{\pi_A} & V(A) \arrow{d}{\pi_{V(A)}} \\
V(A) \arrow{r}{v(\pi_A)} & V(Q(A)) && Q(A) \arrow{r}{Q(v_A)} & Q(V(A))
\end{tikzcd}
\]

Thus we have

\begin{align*}
(g \circ f) \circ \pi_{V(A)} \circ v_A & = g \circ v(\pi_A) \circ v_A \\
                                    & = g \circ v_{Q(A)} \circ \pi_A \\
                                    & = Q(v_A) \circ \pi_A = \pi_{V(A)} \circ v_A.
\end{align*}

Then by the uniqueness of the universal property of $v_A$ and $\pi_{V(A)}$ we have $g \circ f = Id_{Q(V(A))}$. Similar argument shows that $f \circ g = Id_{V(Q(A))}$ and so $V(Q(A)) \cong Q(V(A))$.
\end{proof}

Let $A$ be a ring. Then $V(A)$ is also a ring. In fact, given $f,g \in V(A)$ and $x,y \in f + g$, we have for all $p \in \mbox{spec}(A)$ that $x(p), y(p) \in f(p) + g(p)$ in $K_A(p)$; but since $A$ is a ring, $K_A(p)$ is also a ring and thus $x(p) = y(p)$. Then $x = y$ in $V(A)$. An interesting corollary of the preceding theorem is that if a RRM is representable by a ring using a Marshall quociente, then it is canonically representable by a von Neumann regular ring.

\begin{corollary} \label{represent}
Let $A$ be a semi-real ring. Assume that exists a ring $B$ and $S \subseteq B$ multiplicative subset such that $Q(A) \cong B /_m S$. Then $Q(A) \cong V(A) /_m (1 + \sum V(A)^2)$ is represented by the von Neumann regular ring $V(A)$.
\end{corollary}

\begin{proof}
Since the Marshall quotient preserves hyperring (Proposition $\ref{marshallQ}$), $Q(A) \cong B/_m S$ is a real reduced hyperring and by Proposition $\ref{GvNHR}$ $Q(A)$ is a geometric von Neumann Hyperring. Thus by Theorem $\ref{uniPGvNH}$ and Theorem $\ref{QVVQ}$ we have $Q(A) \cong V(Q(A)) \cong Q(V(A))$. Using the representation given in Theorem $\ref{reprevN}$, the conclusion $Q(A) \cong V(A) /_m (1 + \sum V(A)^2)$ follows.
\end{proof}

\section{Final remarks and future works}

In Theorem $\ref{algV(A)}$, some algebraic properties of von Neumann hull were stated and proved. It is interesting to note that they are enough to characterize the hull. In other words, given a multiring $A$ and a morphism $f \colon A \to V$ where $V$ is a geometric von Neumann hyperring satisfying $i)$ the induced map $\mbox{spec}(V) \to (\mbox{spec}(A))_{const}$ is homeomorphism and $iii)$ for each $p \in \mbox{spec}(V)$, the induced map $K_A(q) \to K_V(p)$, where $q = f^{-1}(p)$, is an isomorphism, then $V \cong V(A)$. Furthermore, the functor $V \colon \textbf{Multi} \to \textbf{GvNH}$ preserves finite product and since it is a left-adjoint it also preserves inverse limits.

In the present work, we prove that von Neumann hull of a RRM is again a RRM (Proposition $\ref{GvnHRRM}$) (von Neumann RRM were studied in \cite{Mir1} through the equivalent notion of von Neumann RS). This will be used to give applications for quadratic forms. More precisely, the concept of Witt ring will be generalized from special groups to real semigroups in \cite{RM1} and we will show that $V(R)$, the von Neumann regular hull of a real semigroup $R$, can also be constructed from $W(R)$, the Witt ring of $R$. This will imply that there is canonical isomorphism between the Witt rings $W(R)$
and $W(V(R))$. 

The above relations allows the analysis of the Witt ring with tools available for RS von Neumann as the description of the isometry of forms through a pp-formula and the characterization
of the transversal representation in terms of isometry. This will be used to provide calculation of the graded  Witt ring, an axiomatization of the Witt rings in a convenient language $L_{\omega_1,\omega}$ and classification of categorical quotients of Witt rings. These results will be presented in \cite{RM1} and \cite{RM2}.




\end{document}